\documentclass[12pt]{amsart}

\usepackage{amsmath,amssymb,latexsym,fullpage,amsfonts,mathrsfs,xfrac,appendix,enumitem,mathabx,epsfig}

\usepackage[alphabetic,nobysame]{amsrefs}

\usepackage[all, cmtip]{xy}

\newtheorem{Theorem}{Theorem}[section] 
\newtheorem{Lemma}[Theorem]{Lemma}     
\newtheorem{Corollary}[Theorem]{Corollary}
\newtheorem{Proposition}[Theorem]{Proposition}

		\newtheorem{Definition}[Theorem]{Definition}
		\newtheorem{Remark}[Theorem]{Remark}
		
		\numberwithin{equation}{section}
		\numberwithin{Theorem}{section}
		\newcommand{\reff}{\eqref}

\newcommand{\defi}[1]{\textsf{#1}} 
\DeclareFontEncoding{OT2}{}{} 
\newcommand{\textcyr}[1]{%
 {\fontencoding{OT2}\fontfamily{wncyr}\fontseries{m}\fontshape{n}\selectfont #1}}
\newcommand{\Sha}{{\mbox{\textcyr{Sh}}}}

\newcommand{\Z}{{\mathbb Z}}
\newcommand{\Q}{{\mathbb Q}}
\newcommand{\R}{{\mathbb R}}

\newcommand{\F}{{\mathbb F}}

\newcommand{\PP}{{\mathbb P}}
\newcommand{\G}{{\mathbb G}}
\newcommand{\M}{{\mathbb{G}^\times}}

\newcommand{\kbar}{{\overline{k}}}
\newcommand{\Cbar}{{\overline{C}}}

\newcommand{\Mbar}{{\overline{M}}}
\newcommand{\Xbar}{{\overline{X}}}
\newcommand{\Ybar}{{\overline{Y}}}

\newcommand{\calA}{{\mathcal A}}

\newcommand{\calD}{{\mathcal D}}

\newcommand{\calJ}{{\mathcal J}}

\newcommand{\calS}{{\mathcal S}}
\newcommand{\calT}{{\mathcal T}}

\newcommand{\fm}{{\mathfrak m}}

\newcommand{\To}{\longrightarrow}

\DeclareMathOperator{\Sel}{Sel}
\DeclareMathOperator{\Cov}{Cov}

\DeclareMathOperator{\disc}{disc}
\DeclareMathOperator{\Tr}{Tr}

\DeclareMathOperator{\divv}{div}
\DeclareMathOperator{\ord}{ord}
\DeclareMathOperator{\coker}{coker}

\DeclareMathOperator{\Gal}{Gal}

\DeclareMathOperator{\Res}{Res}
\DeclareMathOperator{\Br}{Br}

\DeclareMathOperator{\Div}{Div}

\DeclareMathOperator{\Pic}{Pic}

\DeclareMathOperator{\Jac}{Jac}
\DeclareMathOperator{\HH}{H}

\DeclareMathOperator{\Spec}{Spec}
\DeclareMathOperator{\Aut}{Aut}

\DeclareMathOperator{\SL}{SL}

\DeclareMathOperator{\res}{\operatorname{res}}

\title[Generalized Jacobians and explicit descents]
 {Generalized Jacobians and explicit descents} 

\author{Brendan Creutz}

\address{School of Mathematics and Statistics, University of Canterbury, Private Bag 4800, Christchurch 8140, New Zealand}
\email{brendan.creutz@canterbury.ac.nz}

\begin{document}
\maketitle

\begin{abstract}
We develop a cohomological description of explicit descents in terms of generalized Jacobians, generalizing the known description for hyperelliptic curves. Specifically, given an integer $n$ dividing the degree of some reduced, effective and base point free divisor $\frak{m}$ on a curve $C$, we show that multiplication by $n$ on the generalized Jacobian $J_\frak{m}$ factors through an isogeny $\varphi:A_\frak{m} \to J_\frak{m}$ whose kernel is dual to the Galois module of divisor classes $D$ such that $nD$ is linearly equivalent to some multiple of $\frak{m}$. By geometric class field theory, this corresponds to an abelian covering of $C_\kbar := C \times_{\Spec{k}} \Spec(\kbar)$ of exponent $n$ unramified outside $\frak{m}$. We show that the $n$-coverings of $C$ parameterized by explicit descents are the maximal unramified subcoverings of the $k$-forms of this ramified covering. We present applications to the computation of Mordell-Weil ranks of nonhyperelliptic curves.
\end{abstract}

\section{Introduction}

	Suppose $f(x,y)$ is a binary form of degree $d$ over a field $k$ of characteristic not equal to $2$. Pencils of quadrics with discriminant form $f(x,y)$ have been studied in \cite{BSD,Cassels,Cremona,BG, Wang, BGW_AIT2, BGW}. When $d$ is even, the $\SL_d(k)/\mu_2$-orbits of pairs $(A,B)$ with discriminant form $f(x,y)$ correspond to a collection of $2$-coverings of the hyperelliptic curve $C:z^2 = f(x,y)$. When $k = \Q$ these coverings are used in \cite{Bhargava} and \cite{BGW} to compute the average size of the $2$-Selmer set of $C$, and of the torsor $J^1$ parameterizing divisor classes of degree $1$, respectively, from which they deduce the fantastic result that most hyperelliptic curves over $\Q$ have no rational points.
	
	The same collection of coverings can also be described in terms of the $k$-algebra $L := k[x]/f(x,1)$. This description was used in \cite{BruinStoll} and \cite{CreutzANTSX} to compute $2$-Selmer sets of $C$ and $J^1$, respectively, for individual hyperelliptic curves. A key step in both \cite{CreutzANTSX} and \cite{BGW} is to check that this collection of coverings is large enough to contain the locally soluble $2$-coverings (under suitable hypotheses on $C$). In \cite{BGW} this is achieved by identifying these coverings as the unramified subcoverings of $k$-forms of the maximal abelian covering of exponent $2$ unramified outside the pair of points at infinity on the affine model $z^2=f(x,y)$, a characterization that is quite natural in light of the use of generalized Jacobians in \cite{PoonenSchaefer}.
	
	Meanwhile the theory of explicit descents has expanded to incorporate computable descriptions of certain approximations to Selmer sets, called fake Selmer sets, for all curves. This is developed for nonhyperelliptic curves of genus at least $2$ in \cite{BPS} and for curves of genus $1$ in \cite{CreutzMathComp}. In this paper we provide geometric and cohomological descriptions of these descents in terms of generalized Jacobians, generalizing the description for hyperelliptic curves given in \cite{PoonenSchaefer,BGW}. Specifically, given an integer $n$ dividing the degree of some reduced effective and base point free divisor $\frak{m}$ on a curve $C$, we show that multiplication by $n$ on the generalized Jacobian $J_\frak{m}$ factors through an isogeny of semiabelian varieties $\varphi:A_\frak{m} \to J_\frak{m}$ whose kernel is naturally the dual of the Galois module of classes of divisors $D$ on $C_\kbar := C \times_{\Spec{k}} \Spec(\kbar)$ such that $nD$ is linearly equivalent to a multiple of $\frak{m}$. By geometric class field theory, this corresponds to an abelian covering of exponent $n$ and conductor $\frak{m}$. We show that the fake descents mentioned above have a natural interpretation in terms of the $k$-forms of this ramified covering, which we call $\varphi$-coverings. The main result in this direction is Theorem~\ref{thm:descentmaps}, from which we deduce Corollaries~\ref{cor:alphais} and~\ref{cor:alphaC} giving an interpretation of the descents on $C$ and $J^1$ in terms of those $n$-coverings which arise as maximal unramified subcoverings of the $k$-forms of this ramified covering.
	
	This description unifies the methods of explicit descent described in \cite{MSS,BruinStoll,CreutzANTSX,CreutzMathComp,BPS} and allows a more
natural interpretation of some of the objects that arise. Moreover, it yields a number of applications to explicit descent and the arithmetic of curves described in the following subsections.

\subsubsection{Applications to explicit descent on $J$}

	The fake descent presented in \cite{BPS} proceeds by substituting the connecting homomorphism $d : J(k) \to \HH^1(k,J[n])$ in the Kummer sequence with a more computationally amenable homomorphism $f : \Pic^0(C) \to L^\times/k^\times L^{\times n}$, for some \'etale $k$-algebra $L$. Here $\Pic(C)$ denotes the group of $k$-rational divisors on $C$ modulo linear equivalence and $\Pic^0(C)$ denotes the subgroup of classes of degree $0$. In order to obtain information about the Selmer group from this, they require some hypothesis (e.g., \cite[Hypothesis 10.1]{BPS}) to ensure that $\Pic^0(C) = J(k)$ globally and locally. In general one has an injective map $\Pic^0(C) \to \Pic^0(\Cbar)^{\Gal{k}} = J(k)$ which need not be surjective. We show how such hypotheses can be omitted in a number of relevant cases (cf. Theorem~\ref{thm:descentJ}). In Theorem~\ref{thm:descentJexample} we use this to determine the Mordell-Weil rank of a Jacobian $J$ of a plane quartic curve $C$ for which $\Pic^0(C) \ne J(k)$.

\subsubsection{Application to explicit descent on $C$ and $J^1$}

	In \cite{BPS} a `fake Selmer set' of a nonhyperelliptic curve $C$ over a global field is introduced. Using the machinery of \cite{BPS} we introduce a fake Selmer set of the torsor $J^1$ parameterizing divisor classes of degree $1$ (see Definition~\ref{def:selfaked}). It is easy to see that $C$ and $J^1$ cannot have any rational points if the corresponding fake Selmer set is empty. Our interpretation in terms of generalized Jacobians allows us to verify that the obstruction coming from these fake descents is indeed a finite abelian descent obstruction in the sense of~\cite[Section 5.3]{Skorobogatov} and, consequently, that such counterexamples to the Hasse principle are explained by the Brauer-Manin obstruction (cf. \cite[Theorem 6.1.2]{Skorobogatov}). This is given in Theorem~\ref{thm:selfakeC} and~\ref{thm:selfakeJ1} below. 
	
	Particularly in the case of $J^1$, this allows us to obtain deeper knowledge than would otherwise be obtained from simply knowing that $J^1(k)$ is empty. Indeed, we are able to tap into results in arithmetic duality which would otherwise only be possible conditionally on deep conjectures concerning finiteness of the Tate-Shafarevich group $\Sha(J)$. In Section~\ref{sec:Examples} we give an example of a nonhyperelliptic genus $3$ curve over $\Q$ with absolutely simple Jacobian $J$ for which the fake $2$-Selmer set of $J^1$ is empty. Theorem~\ref{thm:selfakeJ1} is then used to prove that $\Sha(J)[2^\infty] \simeq \Z/2\Z\times \Z/2\Z$, and consequently to determine that the Mordell-Weil rank is $1$. Without making use of Theorem \ref{thm:selfakeJ1} we would only obtain the weaker conclusion that $1 \le \operatorname{rank}(J(\Q)) \le 2$ and $1 \le \dim_{\F_2}\Sha(J)[2] \le 2$.
	
\subsubsection{Applications to descent on genus $1$ curves} The results of \cite{CreutzMathComp} describe $n$-descents on genus $1$ curves of degree $n$ when $n$ is prime. The results just mentioned extend aspects of this to general $n$. Namely, for such a curve we have a computable fake Selmer set whose emptiness implies that the curve is not divisible by $n$ in the Tate-Shafarevich group of its Jacobian. This is potentially practical in the case $n = 4$, enabling $16$-descent on elliptic curves.

\subsubsection{Application to Galois descent of unramified abelian coverings of exponent $2$}

	The results of this paper are used in \cite{CreutzPAMS} to prove that if $C$ is an everywhere locally solvable curve of genus $g \ge 2$ over a global field of characteristic different from $2$ and that the Galois action on $J[2]$ is generic (i.e., $\Gal(k(J[2])/k)$ is isomorphic to $S_{2g+2}$ or $\operatorname{GSp}_{2g}(\F_2)$ correspondingly as $C$ is or is not hyperelliptic), then the maximal unramified abelian covering of $C_\kbar$ of exponent $2$ descends to $k$.
	
	The obstruction to Galois descent for the $\varphi$-covering mentioned above and its maximal unramified subcovering are elements of the Galois cohomology groups $\HH^2(k,A_\frak{m}[\varphi])$ and $\HH^2(k,J[2])$, respectively. The proof proceeds by showing that, generically, the locally trivial subgroup $\Sha^2(k,A_\frak{m}[\varphi])$ is trivial, which implies that the ramified covering and, hence also, its unramified subcovering descend to $k$. The use of $\varphi$-coverings here seems unavoidable (and the result all the more surprising) given that the group $\Sha^2(k,J[2])$ can be nontrivial even when the Galois action on $J[2]$ is generic. In fact, this occurs whenever $C$ has no rational theta characterstics and all of the decomposition groups of $\Gal(k(J[2])/k)$ are cyclic, since in this case the torsor parameterizing theta characteristics gives a nontrivial element of $\Sha^1(k,J[2])$ (see \cite{Atiyah} and \cite[Remark 3.18]{PoonenRains}) and $\Sha^2(k,J[2]) \simeq \Sha^1(k,J[2])$ by Tate's duality theorem. Moreover, there are examples of locally solvable curves of genus $\ge 2$ for which the maximal unramified abelian covering of exponent $2$ does not descend to $k$ (see \cite[Theorem 6.7]{CreutzViray}).

\subsubsection{Potential application to average sizes of Selmer sets} We expect our interpretation may also be of relevance to future efforts to compute these Selmer sets {\em on average}. Namely, it should be possible to identify the collection $\Cov_\frak{m}^n(J^1)$ with the orbits in some coregular representation (as is done in \cite{BGW} for the hyperelliptic case). The results in Theorems~\ref{thm:solublecoverings} and~\ref{thm:phicov} would then have implications for the structure of the space of orbits. Thorne has recently made progress understanding the situation for nonhyperelliptic genus $3$ curves with a marked rational point \cite{Thorne1,Thorne2}. It is our hope that the results of this paper may shed light on the corresponding situation when there are no rational points.
		
		\subsection{Notation}
				
			Throughout this paper $n$ is an integer and $k$ is a field of characteristic not divisible by $n$, with separable closure $\kbar$ and absolute Galois group $\Gal_k = \Gal(\kbar/k)$. We will use $C$ to denote a \defi{nice curve} over $k$, i.e. a smooth, projective and geometrically integral $k$-variety of dimension $1$. For a nonempty finite \'etale $k$-scheme $\Delta = \Spec(L)$ we use $\Res_\Delta = \Res_{L/k}$ to denote the restriction of scalars functor taking $L$-schemes to $k$-schemes. For a commutative \'etale $k$-group scheme $G$, we use $\HH^i(G)$ to denote the Galois cohomolgy group $\HH^i(\Gal_k,G(\kbar))$. For $k$ a global field, equivalence classes of absolute values on $k$ (whether archimedean or not) will be referred to as primes.

		\subsection*{Acknowledgements}
			I would like to thank: Michael Stoll and Bjorn Poonen for comments and discussions concerning the material in this article and Nils Bruin for providing me with Magma code for a number of the algorithms described in \cite{BPS}. In developing the algorithms and examples in Section~\ref{sec:examples} I have made use of a list quartic curves of small discriminant provided by Denis Simon as well as the database  \cite{Sutherland} developed by Andrew Sutherland. Computations were performed using the Magma Computer Algebra System described in \cite{magma}.
			\section{The modulus setup}\label{sec:modulus}

				\begin{Definition}
					Let $C$ be a nice curve over $k$. A \defi{modulus setup} for $C$ is a pair $(n,\frak{m})$ consisting of a positive integer $n$ not divisible by the characteristic of $k$, and a reduced, effective and base point free divisor $\frak{m} \in \Div(C)$ of degree $m$, with $n$ dividing $m$.
				\end{Definition}
				
				Given a modulus setup $(n,\frak{m})$ we define $\ell := \deg(\frak{m})/n$. We are primarily interested in the following examples.
					
					\begin{enumerate}[label=\textbf{M.\arabic*},ref=Example M.\arabic*]
						\item\label{ex:M_doublecover} Suppose $\pi:C \to \PP^1$ is a double cover which is not ramified over $\infty \in\PP^1$. Let $n = 2$ and $\frak{m} = \pi^*\infty$. 
						\item\label{ex:M_g1} Suppose $C$ is a genus $1$ curve of degree $m$ in $\PP^{m-1}$. Take $\frak{m}$ to be any reduced hyperplane section and take $n = m$.
						\item\label{ex:M_canonical} Suppose $C$ is any nice nonhyperelliptic curve of genus at least $3$, $n = 2$ and $\frak{m}$ is a canonical divisor. Then $m = 2g-2$ and $\ell = g-1$.
					\end{enumerate}
				  
				\subsection{The generalized Jacobian associated to a modulus setup}
					Let $C$ be a nice curve over $k$ with a modulus setup $(n,\frak{m})$. We may view $\frak{m}$ as a finite \'etale subscheme $\frak{m} = \Spec M \subset C$, or as a \defi{modulus} in the sense of geometric class field theory (see \cite{SerreAGCF}). Let $C_\frak{m}$ denote the singular curve associated to $\frak{m}$ as in \cite[IV.4]{SerreAGCF}. Let $\Pic_C$ and $\Pic_{C_\frak{m}}$ be the commutative group schemes over $k$ representing the Picard functors of $C$ and $C_\frak{m}$. There is an exact sequence of commutative group schemes over $k$,
				\begin{equation}\label{eq:PicCmseq}
					1 \to T \to \Pic_{C_\frak{m}} \to \Pic_C \to 0\,,
				\end{equation}
				where $T$ is an algebraic torus. The restriction of \eqref{eq:PicCmseq} to the identity components is an exact sequence of semiabelian varieties,
				\begin{equation}
					1 \to T \to J_\frak{m} \to J \to 0\,,
				\end{equation}
				where $J_\frak{m}$ is the generalized Jacobian of $C$ associated to the modulus $\frak{m}$ and $J$ is the usual Jacobian of $C$.
				
				Let $\M$ denote the multiplicative group scheme.\footnote{In conjunction with our use of $\frak{m}$ for the modulus and $m$ for its degree, the standard notation $\G_m$ for the multiplicative group might lead to confusion.}

				\begin{Lemma}\label{lem:Tn}
					$T\simeq \Res_\frak{m}\M/\M$ is isomorphic to the quotient of $\Res_\frak{m}\M$ by the diagonal embedding of $\M$, and there is an exact sequence of finite group schemes
					\[
						1 \To \frac{\Res^1_\frak{m}\mu_n}{\mu_n} \To T[n] \To \mu_n \To 1 \,,
					\]
					where $\Res^1_\frak{m}\mu_n$ is the kernel of the norm map $N:\Res_\frak{m}\mu_n \to \mu_n$.
				\end{Lemma}

				\begin{proof}
					The first statement, that $T = \Res_\frak{m}\M/\M$, follows from well known results on the structure of generalized Jacobians (see \cite[\S V Prop. 7]{Serre AGCF}). Let $\Res^1_\frak{m}\M$ denote the kernel of the norm map $\Res_\frak{m}\M \to \M$. The inclusion map $\Res^1_\frak{m}\M \to \Res_\frak{m}\M$ induces a surjective map onto $\Res_{\frak{m}}\M/\M$ with kernel $\mu_m$. This gives the middle rows of the following commutative and exact diagram.
					\[
					\xymatrix{
						& \mu_n \ar[r]\ar@{^{(}->}[d] &\Res_\frak{m}^1\mu_n \ar[r]\ar@{^{(}->}[d] \ar[r] &T[n]\ar@{^{(}->}[d]\\
						1  \ar[r]& \mu_m \ar[r]\ar[d]^n & \Res^1_\frak{m}\M \ar[r]\ar[d]^n & \frac{\Res_\frak{m}\M}{\M} \ar[r]\ar[d]^n & 1\\	
						1  \ar[r]& \mu_m \ar[r]\ar@{->>}[d]^{m/n} & \Res^1_\frak{m}\M \ar[r]\ar[d] & \frac{\Res_\frak{m}\M}{\M} \ar[r] & 1\\
						& \mu_n \ar[r] & 1&\\
					}
					\]
					The exact sequence in the statement of the lemma follows by applying the snake lemma.
				\end{proof}

				\subsection{The isogeny associated to a modulus setup}

				\begin{Lemma}\label{lem:defineG}
					Given a modulus setup $(n,\frak{m})$ there is a commutative group scheme $\frak{A}$ over $k$ and isogenies $\psi:\Pic_{C_\frak{m}} \to \frak{A}$ and $\varphi:\frak{A} \to \Pic_{C_\frak{m}}$ such that $\ker(\psi)= \frac{\Res^1_\frak{m}\mu_n}{\mu_n} \subset T[n]$ and $\varphi\circ\psi = [n]$. Moreover, we have a commutative and exact diagram
					\[
						\xymatrix{
							1 \ar[r]& T' \ar[r]\ar[d]^\varphi & \frak{A} \ar[r]\ar[d]^\varphi & \Pic_C \ar[r]\ar[d]^n & 0\\
							1 \ar[r]& T \ar[r] & \Pic_{C_\frak{m}} \ar[r] & \Pic_C \ar[r] & 0\,.
						}
					\]
					where $T'$ is a torus and $T'[\varphi] \simeq \mu_n$.
				\end{Lemma}

				\begin{proof}
					By Lemma~\ref{lem:Tn}, $\Pic_{C_\frak{m}}$ contains a finite group scheme isomorphic to $\Res^1_\frak{m}\mu_n/\mu_n$. The quotient of $\Pic_{C_\frak{m}}$ by this subgroup scheme yields an isogeny $\psi: \Pic_{C_\frak{m}} \to \frak{A}$. The existence of $\varphi$ follows from the fact that $\ker(\psi)$ is contained in the kernel of multiplication by $n$. Since $\ker(\psi) \subset T$, $\frak{A}$ is an extension of $\Pic_C$. The assertion that $T'[\varphi] \simeq \mu_n$ follows from Lemma~\ref{lem:Tn}.
				\end{proof}		

				\begin{Remark}
					When $n = m = \deg(\frak{m}) = 2$, we have that $T[n] \simeq \mu_n$. Hence $\psi$ is the identity map on $\frak{A} = \Pic_{C_{\frak{m}}}$ and $\varphi$ is multiplication by $2$.
				\end{Remark}

				\subsection{Description using divisor classes}
				A function $f \in k(C)^\times$ that is regular on $\frak{m}$  gives, by restriction, an element $f|_\frak{m} \in M$, where $\Spec(M) = \frak{m}$. We use $\Div_\frak{m}(C)$ to denote the divisors of $C$ that have support disjoint from $\frak{m}$. 

				\begin{Lemma}\label{lem:2.5} Let $\frak{A}$ be as defined in Lemma~\ref{lem:defineG}. Then
					\begin{align*}
						\Pic_{C}(\kbar) 
						&= \Div(C_\kbar)/\{ \divv(f) \;:\; f \in \kbar(C_\kbar)^\times \}\,,\\
						\Pic_{C_\frak{m}}(\kbar) 
						&= \Div_\frak{m}(C_\kbar)/\{ \divv(f) \;:\; f \in \kbar(C_\kbar)^\times,\;f|_\frak{m} = 1\}\,,\\
						\frak{A}(\kbar) 
						&= \Div_\frak{m}(C_\kbar)/\{ \divv(f) \;:\; f \in \kbar(C_\kbar)^\times,\;f|_\frak{m}  \in \Res^1_\frak{m}\mu_n\}\,.
					\end{align*}
					Moreover, $\varphi : \frak{A} \to \Pic_{C_\frak{m}}$ is induced by multiplication by $n$ on $\Div_\frak{m}(C_\kbar)$.
				\end{Lemma}	

				\begin{proof}
					The first two statements are well known (see \cite{SerreAGCF}; note that $f|_\frak{m}  =1$ if and only if $f \equiv 1\bmod \frak{m}$, since $\frak{m}$ is reduced). The $\kbar$-points of the subgroup $T = \Res_\frak{m}\M/\M \subset \Pic_{C_\frak{m}}$ are represented by divisors of functions which do not vanish on $\frak{m}$,
				\[
					T(\kbar) = \frac{
						\{ \divv(f) \;:\; f \in \kbar(C_\kbar)^\times,\;f|_\frak{m}  \in \Mbar^\times \}}{
						\{ \divv(f) \;:\; f \in \kbar(C_\kbar)^\times,\;f|_\frak{m}  = 1\}}
						\,.
				\]
				The description of $\frak{A}(\kbar)$ in the statement then follows from the fact that $\frak{A}$ is the quotient of $\Pic_{C_\frak{m}}$ by the image of $\Res_\frak{m}^1\mu_n$ in $T$. The final statement follows easily from the fact that $\varphi\circ\psi$ is multiplication by $n$ on $\Pic_{C_\frak{m}}$.
				\end{proof}

				\subsection{Component groups}

				The component groups of $\Pic_{C_\frak{m}}, \Pic_C$ and $\frak{A}$ are all isomorphic to $\Z$, the isomorphism being given by the degree map on divisor classes. The degree $0$ component of $\frak{A}$ is a semiabelian variety $A_\frak{m}$ fitting into an exact sequence,
				\begin{equation}\label{eq:Amsemiabelian}
					1 \to T' \to A_\frak{m} \to J \to 0\,.
				\end{equation}
				In particular, $A_\frak{m}$ is geometrically connected.
				We label the components
				\begin{equation}\label{label:components}
					\Pic_C = \bigsqcup_{i \in \Z} J^i\,, \quad \quad
					\Pic_{C_\frak{m}} = \bigsqcup_{i \in \Z} J_\frak{m}^i\,,\quad\quad 										\frak{A}= \bigsqcup_{i \in \Z}A_\frak{m}^i\,,
				\end{equation}
				so that the superscripts denote the image under the degree map. To ease notation we also denote the degree $0$ components by $J = J^0$, $J_\frak{m}=J_\frak{m}^0$ and $A_\frak{m} = A^0_\frak{m}$. For any $i \in \Z$, $J^i$ and $J^i_\frak{m}$ are torsors under $J$ and $J_\frak{m}$, respectively.

				Let $\frak{m}' \in \Div_\frak{m}(C)$ be an effective reduced $k$-rational divisor linearly equivalent to and with disjoint support from $\frak{m}$ (which exists by Bertini's theorem, provided $k$ has sufficiently many elements). Then $\frak{m}'$ determines classes in $\Pic_C$, $\Pic_{C_\frak{m}}$ and $\frak{A}$, which generate, in each, a subgroup scheme isomorphic to the constant group scheme $\Z$. Let $\calJ$, $\calJ_\frak{m}$ and $\calA_\frak{m}$ denote the corresponding quotient group schemes, which exist since the category of commutative algebraic groups is abelian. We have,
				\begin{equation}\label{eq:choosed}
					\calJ := \frac{\Pic_C}{\Z  \frak{m}'} = \bigsqcup_{i = 0}^{m-1} J^i\,, \quad \quad
					\calJ_\frak{m} := \frac{\Pic_{C_\frak{m}}}{\Z  \frak{m}'} = \bigsqcup_{i = 0}^{m-1} J_\frak{m}^i\,,\quad\quad 								
					\calA_\frak{m} := \frac{\frak{A}}{\Z  \frak{m}'} = \bigsqcup_{i = 0}^{m-1}A_\frak{m}^i\,,
				\end{equation}
				where we have abused notation slightly by writing $\frak{m}'$ to also denote its class in $\Pic_C$, $\Pic_{C_\frak{m}}$ and $\frak{A}$, respectively. 

				\begin{Remark}
 It is not generally true that all effective divisors linearly equivalent to and disjoint from $\frak{m}$ give the same class in $\Pic_{C_\frak{m}}$, so the quotient maps $\Pic_{C_\frak{m}} \to \calJ_\frak{m}$ and $\frak{A} \to \calA_\frak{m}$ may depend on the choice for $ \frak{m}'$. However, the map $\Pic_C \to \calJ$ depends only on $\frak{m}$.
				\end{Remark}

The maps $\psi$ and $\varphi$ of Lemma~\ref{lem:defineG} induce maps $\psi: \calJ_\frak{m} \to \calA_\frak{m}$ and $\varphi: \calA_\frak{m} \to \calJ_\frak{m}$ whose composition is multiplication by $n$. The map $\varphi$ induces a morphism of exact sequences of group schemes,
				\begin{equation}\label{eq:morph2}
					\xymatrix{
						0 \ar[r]& T' \ar[r]\ar[d]^{\varphi}& \calA_\frak{m}\ar[d]^\varphi\ar[r]& \calJ \ar[r]\ar[d]^n& 0\,\phantom{.}\\
						0 \ar[r]& T  \ar[r]& \calJ_\frak{m} \ar[r]& \calJ \ar[r]& 0\,,
					}
				\end{equation}
				and in particular an isogeny of semiabelian varieties,
				\begin{equation}\label{eq:definephi}
					\varphi \colon A_\frak{m} \To J_\frak{m}\,.
				\end{equation}
				\begin{Lemma}\label{lem:torsion1} There is a commutative and exact diagram,
					\begin{equation}\label{diag:torsion1}
							\xymatrix{
								\mu_n \ar@{^{(}->}[d] \ar@{=}[r] & \mu_n \ar@{^{(}->}[d]\\
								A_\frak{m}[\varphi] \ar@{^{(}->}[r] \ar@{->>}[d] 
								& \calA_\frak{m}[\varphi] \ar@{->>}[r]^-{\frac{1}{\ell}\deg} \ar@{->>}[d]& \Z/n\Z\ar@{=}[d]\\
								 J[n] \ar@{^{(}->}[r] & \calJ[n]  \ar@{->>}[r]^-{\frac{1}{\ell}\deg} & \Z/n\Z \,.
							}
					\end{equation}
				\end{Lemma} 

\begin{proof}
	The first and second columns are, respectively, the kernels of the morphism of exact sequences appearing in Lemma~\ref{lem:defineG} and Diagram~\eqref{eq:morph2}. They are exact by the snake lemma, since $\varphi:T' \to T$ is an epimorphism. By Lemma~\ref{lem:2.5}, a divisor $D \in \Div_\frak{m}(C_\kbar)$ represents a class in $\calA_\frak{m}[\varphi]$ if and only if $nD = a\frak{m}' + \divv(f)$, for some $a \in \Z$ and $f \in k(C_\kbar)^\times$ with $f|_\frak{m}  \in \Res^1_\frak{m}\mu_n$. In particular, $n \deg(D) = \deg(nD) = \deg(a\frak{m}') = an\ell$. So $\ell$ divides $\deg(D)$. As every class in $\calA_\frak{m}$ can be represented by a divisor of degree $1 \le d \le m$ this shows that the maps in the first row are well defined. By definition, $A_\frak{m}[\varphi]$ is the intersection of the kernels of the maps $\varphi$ and $\deg$ on $\calA_\frak{m}$. Surjectivity in the middle row follows from the fact that $A_\frak{m}$ is a divisible group. Namely, there exists $\frak{n}' \in \Div(\Cbar)$, necessarily of degree $\ell$, such that the class of $n\frak{n}'$ is equal to that of $\frak{m}'$ in $\frak{A}(\kbar)$. By Lemma~\ref{lem:2.5}, $\varphi([\frak{n}']) = [n\frak{n'}] = [\frak{m}'] = 0$ in $\mathcal{A}_\frak{m}$. Thus the middle row is exact. The same argument (applied to $\calJ$ in place of $\calA_\frak{m}$) shows the same for the bottom row.
\end{proof}

				\subsection{Extended Weil pairings} 

					We now define a bilinear pairing
					\[
						e : \calJ_\frak{m}[n] \times \calJ_\frak{m}[n] \to \mu_n\,.
					\]
					Fix $f \in k(C)^\times$ such that $\divv(f) = \frak{m}'-\frak{m}$. Given $\calD_1,\calD_2 \in \calJ_\frak{m}[n]$, choose representative divisors $D_1,D_2 \in \Div_\frak{m}(C_\kbar)$, and let $d_i = \deg(D_i)/\ell$, where we remind the reader that $\ell := m/n$. There exist unique functions $h'_i \in k(C_\kbar)^\times$ such that $nD_i = \divv(h_i') + d_i\frak{m'}$ and $h'_i|_\frak{m} = 1$. Set $h_i = f^{d_i}h'_i$, so that $nD_i = \divv(h_i) + d_i\frak{m}$. Define:
					\begin{equation}\label{eq:WPdef}
						e(\calD_1,\calD_2) = (-1)^{d_1d_2}\prod_{P \in C(\kbar)}(-1)^{n(\ord_PD_1)(\ord_PD_2)}\frac{h_2^{\ord_PD_1}}{h_1^{\ord_PD_2}}(P) \in \kbar^\times\,.
					\end{equation}
					
					We note that when $D_1$ and $D_2$ have disjoint support, this can be written as
					\begin{equation}
						e(\calD_1,\calD_2) = (-1)^{d_1d_2}\frac{h_2(D_1)}{h_1(D_2)}\,.
					\end{equation}

					\begin{Proposition}\label{prop:WP2}
						The pairing $e$ is Galois equivariant and induces, via the surjective map $\psi:\calJ_{\frak{m}}[n] \to \calA_\frak{m}[\varphi]$ and the maps in ~\eqref{diag:torsion1}, nondegenerate Galois equivariant pairings
						\begin{align*}
							e&:\calA_\frak{m}[\varphi]\times\calA_\frak{m}[\varphi] \to \mu_n\,,\\
							e&:A_\frak{m}[\varphi]\times\calJ[n] \to \mu_n\,,\\
							e&: J[n]\times J[n] \to \mu_n\,.
						\end{align*}
						Moreover, the pairing on $J[n]\times J[n]$ coincides with the Weil pairing.
					\end{Proposition}

					\begin{Remark}
						The definition of $e$ given above depends on the choice of $\frak{m}'$ in~\eqref{eq:choosed} and the function $f$ with $\divv(f) = \frak{m}'-\frak{m}$. However, as shown in the proof below, the induced pairings on $A_\frak{m}[\varphi]\times \calJ[n]$ and $J[n]\times J[n]$ do not depend on these choices.
					\end{Remark}

				\begin{proof}
					One can check that the pairing $e:\calJ_\frak{m}[n] \times \calJ_\frak{m}[n] \to \mu_n$ is Galois equivariant  exactly as is done in \cite[Section 7]{PoonenSchaefer} where the situation of~\ref{ex:M_doublecover} is considered (one need only replace the function $x$ there with the function $f$ in the definition above).

						We will show that the orthogonal complements of $\Res^1_\frak{m}\mu_n/\mu_n$ and $T[n]$ with respect to $e$ are $\calJ_\frak{m}[n]$ and $J_\frak{m}[n]$, respectively. This is enough to ensure that $e$ induces the pairings stated. The pairing induced on $J[n]$ is evidently the Weil pairing (see \cite[Theorem 1]{Howe}), which is known to be nondegenerate. Nondegeneracy of the other pairings follows from the exactness of~\eqref{diag:torsion1}. Alternatively, Lemma~\ref{lem:WP} below gives an alternative description of this pairing using class field theory which is readily seen to be nondegenerate.
						
						Let $\calD_1 \in T[n]$. Then $\calD_1$ is represented by $D_1 = \divv(f)$ for some $f \in k(C_\kbar)^\times$ with $f|_\frak{m}  \in \Res_{\frak{m}}\mu_n$. Since $nD_1 = \divv(f^n)$ and $f^n|_\frak{m} = 1$ we must use $h_1 =f^n$ in the definition of the pairing. Suppose $\calD_2 \in \calJ[n]$ and let $D_2,h_2,d_2$ be as in the definition of the pairing. Then we have
						\begin{align*}
							e(\calD_1,\calD_2) &= \prod_{P \in C(\kbar)}(-1)^{n(\ord_Pf)(\ord_PD_2)}\frac{h_2^{\ord_Pf}}{f^{n\ord_PD_2}}(P)\\
							&=  \prod_{P \in C(\kbar)}(-1)^{(\ord_Pf)(\ord_Ph_2 + d_2\ord_P\frak{m})}\frac{h_2^{\ord_Pf}}{f^{\ord_Ph_2+d_2\ord_P\frak{m}}}(P) \quad\quad \text{(since $nD_2 = \divv(h_2)+d_2\frak{m}$.)}\\
							&=  \prod_{P \in C(\kbar)}(-1)^{d_2(\ord_Pf)(\ord_P\frak{m})}f^{-d_2\ord_P\frak{m}}(P) \quad\quad \text{(by Weil reciprocity)}\\
							&=  \prod_{P \in C(\kbar)}f^{-d_2\ord_P\frak{m}}(P) \quad\quad \text{(since $f|_\frak{m} $ is invertible)}\\
							&= N(f|_\frak{m} )^{-d_2}\,,
						\end{align*}
						where $N$ denotes the induced norm $\Res_\frak{m}\M \to \M$. From this one easily sees that $\Res^1_\frak{m}\mu_n/\mu_n$ lies in the kernel of the pairing and that $T[n]$ pairs trivially with the degree $0$ subgroup, $J_\frak{m}[n] \subset \calJ_\frak{m}[n]$.
					\end{proof}

				Taking Galois cohomology of~\eqref{diag:torsion1} yields a commutative and exact diagram
					\begin{equation}\label{eq:maindiagram}
						\xymatrix{
							&k^\times/k^{\times n} \ar[d] \ar@{=}[r] & k^\times/k^{\times n}\ar[d]\\
							\Z/n\Z \ar[r]^{\partial'}\ar@{=}[d]
							& \HH^1(A_\frak{m}[\varphi]) \ar[r] \ar[d] 
							& \HH^1(\calA_\frak{m}[\varphi]) \ar[r]^-{\frac{1}{\ell}\deg} \ar[d] 
							& \HH^1(\Z/n\Z) \ar@{=}[d]\\
							\Z/n\Z \ar[r]^{\partial}& \HH^1(J[n]) \ar[r]\ar[d]^\Upsilon & \HH^1(\calJ[n]) \ar[r]^-{\frac{1}{\ell}\deg}\ar[d]^{\Upsilon'} & \HH^1(\Z/n\Z)\\
							&\Br(k)[n] \ar@{=}[r]& \Br(k)[n]
						}
					\end{equation}
								
					\begin{Lemma}\label{lem:WPcupprod}
						 The images of $\partial(1)$ and $\partial'(1)$ in $\HH^1(J)$ and $\HH^1(A_\frak{m})$ are the classes of $J^\ell$ and $A_\frak{m}^\ell$, respectively. The maps $\Upsilon$ and $\Upsilon'$ are given by
						 \begin{align*}
						 	\Upsilon(\xi) &= \xi \cup_e \partial(1)\,, \text{ and}\\
						 	\Upsilon'(\xi) &= \xi \cup_e \partial'(1)\,,
						 \end{align*}
						 where $\cup_e$ denotes the cup product pairings
						 \begin{align*}
							\cup_e &: \HH^1(J[n]) \times \HH^1(J[n]) \to \HH^{2}(\mu_n) = \Br(k)[n]\,, \text{ and} 	\\
							\cup_e &: \HH^1(\mathcal{J}[n]) \times \HH^1(A_{\frak{m}}[\varphi]) \to \HH^{2}(\mu_n) = \Br(k)[n]
					\end{align*}
						determined by the $e$-pairings of Proposition~\ref{prop:WP2} (cf. \cite[page 38]{nsw}).
					\end{Lemma}

			\begin{proof}
			At the level of cocycles, $\partial(1)$ is represented by $\Gal_{k} \ni \sigma \mapsto [\sigma(D) - D] \in J[n]$, where $D \in \Div(\Cbar)$ is such that $nD$ is linearly equivalent to $\mathfrak{m}'$ and the square parentheses denote the class of a divisor in $\Pic(\Cbar)$. The divisor $D$ necessarily has degree $m/n = \ell$, so the image of this cocycle in $\HH^1(J)$ represents $J^\ell$. The claim that $\partial'(1)$ represents $A_\frak{m}^\ell$ is established similarly.
	
			The $e$-pairings of Proposition~\ref{prop:WP2} give commutative diagrams of pairings
	\begin{equation}\label{eq:pairings}
		\begin{array}[c]{cccccc cccccc}
			\mu_n & \times & \Z/n\Z & \rightarrow & \mu_n  &\phantom{\dots\dots}&& \mu_n & \times & \Z/n\Z & \rightarrow & \mu_n\\
			
			\rotatebox{90}{$\hookleftarrow$} && \rotatebox{90}{$\twoheadrightarrow$}&& \rotatebox{90}{$=$} &&& \rotatebox{90}{$\hookleftarrow$} && \rotatebox{90}{$\twoheadrightarrow$}&& \rotatebox{90}{$=$}\\
			A_{\frak{m}}[\varphi] & \times & \calJ[n] & \rightarrow & \mu_n &&& \mathcal{A}_{\frak{m}}[\varphi] & \times & \mathcal{A}_\frak{m}[\varphi] & \rightarrow & \mu_n\\
			\rotatebox{90}{$\twoheadleftarrow$} && \rotatebox{90}{$\hookrightarrow$}&& \rotatebox{90}{$=$}&&& \rotatebox{90}{$\twoheadleftarrow$} && \rotatebox{90}{$\hookrightarrow$}&& \rotatebox{90}{$=$}\\
			J[n]& \times & J[n] & \rightarrow & \mu_n &&& \mathcal{J}[n]& \times & A_\frak{m}[\varphi] & \rightarrow & \mu_n\,.
		\end{array}
	\end{equation}
		Since the maps $\Upsilon$ and $\Upsilon'$ are coboundary maps from the first columns and the maps $\partial$ and $\partial'$ are coboundary maps from the second columns we may apply \cite[Corollary 1.4.5]{nsw} once to each of the diagrams in~\eqref{eq:pairings} to deduce that $\Upsilon(\xi) = \xi \cup_e \partial(1)$ and $\Upsilon'(\xi) = \xi \cup_e \partial'(1)$.

		\end{proof}
		
		\subsection{Brauer class of a $k$-rational divisor class}
			
			Given a nice curve $C$, there is a well known exact sequence
			\begin{equation}\label{eq:defThetaC}
				0 \to \Pic(C) \To \Pic_C(k) \stackrel{\Theta_C}\To \Br(k)\,
			\end{equation}
			(see \cite{Lichtenbaum}). The map $\Theta_C$ gives the obstruction to a $k$-rational divisor class being represented by a $k$-rational divisor.
		
		\begin{Lemma}\label{lem:upsiloniota}
			Let $d:J(k) \to \HH^1(J[n])$ denote the connecting homomorphism in the Kummer sequence. For any $x \in J(k)$ we have $\Upsilon\circ d(x) = \ell\cdot\Theta_C(x)$.
		\end{Lemma}
		
		\begin{proof}
			The image of $d$ is isotropic with respect to the Weil-pairing cup product $\cup_e$. This gives a commutative diagram of pairings
	\begin{equation*}\label{eq:LichtenbaumPairings}
	\begin{array}[c]{cccccc}
		\cup_e: &\HH^1(J[n]) & \times & \HH^1(J[n]) & \rightarrow & \Br(k) \\
		&{\tiny d}\;\rotatebox{90}{$\rightarrow$} && \rotatebox{90}{$\leftarrow$}&& \rotatebox{90}{$=$}\\
		\langle\,,\,\rangle: &J(k)& \times & \HH^1(J) & \rightarrow & \Br(k)
	\end{array}
	\end{equation*}
		By a result of Lichtenbaum (see the proof of \cite[Corollary 1]{Lichtenbaum}) we have that $\langle x,[J^1]\rangle = \Theta_C(x)$. By the previous lemma we have
		\[
			\Upsilon\circ d(x) = d(x) \cup_e \partial(1) = \langle x,[J^\ell]\rangle= \ell\cdot\langle x,[J^1]\rangle = \ell\cdot\Theta_C(x)\,.
		\]
		\end{proof}
		
		Let $J(k)_\bullet$ denote the kernel of the composition $\Upsilon\circ d : J(k) \to \HH^1(J[n]) \to \Br(k)$. Then $\Pic^0(C) \subset J(k)_\bullet \subset J(k)$ and, in general, any of these containments can be proper. Lemma~\ref{lem:upsiloniota} shows that $\Pic^0(C) = J(k)_\bullet$ when $\ell = 1$ (e.g., for the modulus setups of \ref{ex:M_doublecover} and \ref{ex:M_g1}) while the following corollary shows that $J(k)_\bullet = J(k)$ when $k$ is a local or global field and $C$ has a modulus setup as in~\ref{ex:M_canonical}.
					
		\begin{Corollary}\label{Cor:Ups}
			If 
			\begin{enumerate}
				\item\label{ups1} the period of $C$ divides $\ell$, or
				\item\label{ups2} $k$ is a local or global field and $\gcd(m,g-1)$ divides $\ell$,
			\end{enumerate}
			then $\Upsilon \circ d = 0$.
		\end{Corollary}

		\begin{proof}
			The image of $\Theta_C : J(k) \to \Br(k)$ is isomorphic to the cokernel of $\Pic^0(C) \to J(k)$, which is annihilated by the period of $C$ (\cite[Prop. 3.2]{PoonenSchaefer}). Over a local field, the period of $C$ divides $g-1$ (\cite[Prop. 3.4]{PoonenSchaefer}). Since the period also divides $m = \deg(\frak{m})$, \reff{ups2} implies that $\ell$ is divisible by the period locally. Hence $\Upsilon \circ d = 0$ locally. This must also be true globally by the local-global principle for $\Br(k)$.
		\end{proof}	

		We recall that the situation for $\Pic_{C_\frak{m}}$ is different.
		
		\begin{Lemma}\label{lem:Jipoints}
			The natural map $\Div_\frak{m}(C) \to \Pic_{C_\frak{m}}(k)$ is surjective. In particular, for any $i \ge 1$, $\Div^i(C) \ne \emptyset$ if and only if $J^i_\frak{m}(k) \ne \emptyset$.
		\end{Lemma}
		
		\begin{proof}
			The first statement follows from \cite[Lemma 3.5]{PoonenSchaefer}. The second follows from the first by the moving lemma.
		\end{proof}

		\section{$n$-coverings, $\varphi$-coverings and the descent setup}\label{sec:phicoverings}
	
	\subsection{$n$-coverings and $\varphi$-coverings}
		\begin{Definition}
			Suppose $\phi:A\to B$ is an isogeny of semiabelian varieties over $k$ and $V$ is a $B$-torsor. We say $\pi:V' \to V$ is a \defi{$\phi$-covering of $V$} if there exist isomorphisms $a,b$ of $\kbar$-varieties such that $a$ is compatible with the torsor structure of $V$, fitting into a commutative diagram
			\[
				\xymatrix{
					V'_\kbar \ar[r]^b \ar[d]^\pi & A_\kbar \ar[d]^\phi\\
					V_\kbar \ar[r]^a & B_\kbar\,.
				}
			\]
			Let $\Cov^\phi(V)$ denote the set of isomorphism classes of $\phi$-covering of $V$, considered as objects in the category of $V$-schemes. 
		\end{Definition}
		
		To say that $a$ is compatible with the torsor structure means that $a(x+y) = a(x) + y$. Note that the isomorphism $b$ endows $V'$ with the structure of a torsor under $A$ by the rule $x+y = b^{-1}(b(x) + y)$. The classes of these torsors satisfy $\phi_* [V'] = [V] \in \HH^1(B)$. When nonempty, $\Cov^\phi(V)$ is a principal homogeneous space for the group $\HH^1(\ker(\phi))$ acting by twisting. The isogenies $\varphi : A_\frak{m} \to J_\frak{m}$ and $n : J \to J$ give distinguished points in $\Cov^\varphi(J_\frak{m})$ and $\Cov^n(J)$, endowing these sets with the structure of an abelian group and isomorphisms to $\HH^1(A_\frak{m}[\varphi])$ and $\HH^1(J[n])$, respectively.
		
			Suppose $(n,\frak{m})$ is a modulus setup for a nice curve $C$ over $k$. The isogenies $\varphi:A_\frak{m} \to J_\frak{m}$ and $n:J \to J$ give rise to the notions of $\varphi$-coverings of $J^i_\frak{m}$ and $n$-coverings of $J^i$ for each $i \ge 0$. The pullback of a $\varphi$-covering $V \to J^1_\frak{m}$ along the canonical map $(C-\frak{m}) \to J^1_\frak{m}$ sending a geometric point $x$ to the class of the divisor $x$ in $J^1_\frak{m}(\kbar) \subset \Pic_{C_\frak{m}}(\kbar)$ yields an unramified covering of $(C-\frak{m})$. Corresponding to this is a unique (up to isomorphism) morphism $\pi:Y \to C$ of smooth projective curves over $k$ which is unramifed outside $\frak{m}$.

		\begin{Definition}
			Suppose $(n,\frak{m})$ is a modulus setup for a nice curve $C$ over $k$. A morphism $\pi: Y \to C$ of nice curves is a \defi{$\varphi$-covering of $C$} if it is the unique extension of the pullback of a $\varphi$-covering of $J^1_\frak{m}$ along the canonical map $(C-\frak{m}) \to J^1_\frak{m}$. A morphism $\pi:X\to C$ is an \defi{$n$-covering of $C$} if it is the pullback of an $n$-covering of $J^1$ along the canonical map $C \to J^1$. Let $\Cov^n(C)$ and $\Cov^\varphi(C)$ denote, respectively, the sets of isomorphism classes of $n$-coverings and $\varphi$-coverings of $C$ (considered as objects in the category of $C$-schemes).
		\end{Definition}
	
		\begin{Proposition}\label{prop:cft}
			An $n$-covering of $C$ is a $k$-form of the maximal unramified abelian covering of $C$ of exponent $n$. A $\varphi$-covering of $C$ is an abelian covering of exponent $n$ and conductor $\frak{m}$ whose maximal unramified subcovering is an $n$-covering.
		\end{Proposition}
		
		\begin{proof}
			
			Any unramified abelian extension of $\kbar(\Cbar)$ of exponent $n$ is obtained by adjoining $n$-th roots of functions $f \in \kbar(\Cbar)$ with $\divv(f) = nD \in n\Div(\Cbar)$. For any such function, the class of the divisor $D$ lies in $J[n]$. For the (unique up to isomorphism) $n$-covering $\pi :C' \to \Cbar$ we have $J[n](\kbar) = \ker( \pi^* : \Pic^0(\Cbar) \to \Pic^0(C'))$. Thus $\kbar(C')$ contains $n$-th roots of all functions $f$ as above. This proves the first statement.

			Similarly, the field extension of $\kbar(C_\kbar)$ corresponding to a $\varphi$-covering is the compositum of the extensions corresponding to the index $n$ subgroups of $A_\frak{m}[\varphi]$, or equivalently, to the points of order $n$ in the Cartier dual $\calJ[n]$ (the duality is given by Proposition~\ref{prop:WP2}). If $D \in \Div(C_\kbar)$ represents a point of order $n$ in $\calJ[n]$, then there exists a function $h_D \in \kbar(C_\kbar)^\times$ such that $\divv(h_D) = nD - d\frak{m}$ for some $d \in \Z$. The corresponding extension of $\kbar(C_\kbar)$ is obtained by adjoining an $n$-th root of $h_D$. Such extensions are of conductor $\frak{m}$. The maximal unramified subextension is obtained by adjoining $n$-th roots only of those $h_D$ for which $\divv(h_D)-nD = d\frak{m}$ with $d \equiv 0 \bmod n$. These correspond to points in $J[n]$ showing that the maximal unramified subcover of a $\varphi$-covering is an $n$-covering.

		\end{proof}

		When nonempty, the sets $\Cov^n(C)$ and $\Cov^\varphi(C)$ are principal homogeneous spaces for $\HH^1(J[n])$ and $\HH^1(A_\frak{m}[n])$, respectively, acting by twisting. By geometric class field theory the canonical pullback maps $p:\Cov^n(J^1) \to \Cov^n(C)$ and $p_\frak{m}:\Cov^\varphi(J^1_\frak{m}) \to \Cov^\varphi(C)$ are bijections that are equivariant for the actions by $\HH^1(J[n])$ and $\HH^1(A_\frak{m}[n])$. By Proposition~\ref{prop:cft} there is a canonical map $u:\Cov^\varphi(C) \to \Cov^n(C)$, which associates to a $\varphi$-covering of $C$ the maximal unramified intermediate covering of $C$. Let $\Cov_\frak{m}^n(C)$ denote the image of $u$. 
		
		Given a $\varphi$-covering $\pi: F_\frak{m} \to J^i_\frak{m}$, the torsor structures restrict to actions of the tori $T' \subset A_\frak{m}$ and $T \subset J_\frak{m}$ on $F_\frak{m}$ and $J^{i}_\frak{m}$, respectively. The quotients $F = F_\frak{m}/T'$ and $J^i = J^{i}_\frak{m}/T$ are torsors under $J = A_\frak{m}/T' = J_\frak{m}/T$. The existence of these quotients in the category of varieties follows from \cite[Theoreme 7.2]{Grothendieck}, while the induced torsor structure can be established as in the proof of \cite[Lemma 3.1]{Borovoi}. Since the actions of $T'$ and $T$ are equivariant with respect to $\pi$, there is an induced map $\pi':F \to J^i$ which is a torsor under $A_\frak{m}[\varphi]/T'[\varphi] = J[n]$. This induces a map $q: \Cov^\varphi(J^{i}_\frak{m}) \to \Cov^n(J^i)$. Let $\Cov_\frak{m}^n(J^i)$ denote the image of $q$. We record the following.
		
		\begin{Lemma}\label{lem:defm}
			\phantom{.}\hfill
			\begin{enumerate}
				\item The isomorphism $\Cov^n(J) \simeq \HH^1(J[n])$ restricts to $\Cov_\frak{m}^n(J) \simeq \ker(\Upsilon)$.
				\item The maps defined above fit into a commutative diagram,
					\[ \xymatrix{
							\Cov^\varphi(J^1_\frak{m}) \ar[r]^{p_\frak{m}} \ar[d]^q & \Cov^\varphi(C) \ar[d]^{u} \\
							\Cov^n(J^1) \ar[r]^p & \Cov^n(C)\,.
						}
					\]
					In particular, $p$ restricts to give a bijection $p: \Cov^n_\frak{m}(J) \to \Cov^n_\frak{m}(C)$.
			\end{enumerate}
		\end{Lemma}
		
		\begin{proof}
			The first statement follows from exactness in~\eqref{eq:maindiagram}. The second follows from the universal property of the fibered product.
		\end{proof}

		Using $\varphi$-coverings we can give an alternative description of the $e$-pairing on $A_\frak{m}[\varphi]\times \calJ[n]$ given in Proposition~\ref{prop:WP2}. Let $(Y,\pi)$ be a $\varphi$-covering of $C_\kbar$. Let $M = \kbar(Y)$ and $K = \kbar(C_\kbar)$, which we identify with the subfield $\pi^*(K) \subset M$. There are canonical isomorphisms $r:A_\frak{m}[\varphi] \simeq \Gal(M/\pi^*K)$ and $s:\calJ[n] \simeq (K^{\times}\cap M^{\times n})/K^{\times n}$. The latter sends the class of a divisor $D$ to the class of a function $h \in K$ such that $\divv(h) = nD - d\frak{m}$, for some integer $d$. Kummer theory gives a bilinear pairing $\kappa :\Gal(M/K) \times (K^{\times}\cap M^{\times n})/K^{\times n} \to \mu_n.$

		\begin{Lemma}\label{lem:WP}
			For $\calD_1 \in A_\frak{m}[\varphi]$ and $\calD_2 \in \calJ[n]$ we have $e(\calD_1,\calD_2) = \kappa(r(\calD_1),s(\calD_2))$.
		\end{Lemma}

		\begin{proof}
			The analogous statement for the induced pairing on $J[n]\times J[n]$ is the main result of \cite{Howe}. As described in Section 4 of op. cit. it suffices to prove the statement when $k$ is a finite field. Let $\calD_1 \in A_\frak{m}[\varphi](\kbar)$ and $\calD_2 \in \calJ[n](\kbar)$. By possibly enlarging $k$ if necessary we can arrange that the $\calD_i$ are represented by $k$-rational divisors $D_i$ and, moreover, that $J_\frak{m}[n](k) = J_\frak{m}[n](\kbar)$. Take $g \in k(C)^\times$ such that $\divv(g) = nD_2 - d\frak{m}$. Then, as seen in the proof of Proposition~\ref{prop:cft}, $g \in M^{\times n}$. Let $F : J_\frak{m} \to J_\frak{m}$ be the $k$-Frobenius. Then $J_\frak{m}[n](k) \subset \ker(F-1)$, so $F-1$ factors through multiplication by $n$, and hence through $\varphi$. Moreover, the extension $M/K$ extends to a Galois extension $N/K$ with $\Gal(N/K) \simeq J_\frak{m}[F-1] \simeq J_\frak{m}(k)$. 	
			
			All of this fits into a commutative diagram
			\[
				\xymatrix{
				J_\frak{m}(k) \ar[rr]^-\simeq \ar@{->>}[dr]^{(F-1)/\varphi} \ar@{->>}[d]_-{(F-1)/n}&& \Gal(N/K) \ar[d]\\
				J_\frak{m}(k)[n] \ar@{->>}[r]^\psi & A_\frak{m}[\varphi](k) \ar[d]_-{\kappa(r(\cdot),s(\calD_2))} \ar[r]^-\simeq_-r & \Gal(M/K) \ar[d]\\
				&\mu_n(k) & \Gal(K(g^{1/n})/K) \ar[l]^-\simeq\,. }
			\]
			The map from the top left to the bottom right is given by the Artin map of class field theory and, hence, the composition $I^0(k) \to J_\frak{m}(k) \to \mu_n(k)$ from the $k$-id\`eles of $K$ to $\mu_n(k)$ can be computed with Hilbert norm residue symbols (see \cite[\S6.30, p. 150]{SerreAGCF}). Take $a \in I^0(k)$ to be an id\`ele whose divisor class is equal to the class of $D_1$ in $J_\frak{m}(k)$ and $b \in I^0(k)$ such that $(F-1)/n [b] = [a]$ in $J_\frak{m}(k)$. To prove the lemma amounts to checking that $e(\calD_1,\calD_2)$ is equal to the product of the Hilbert norm residue symbols, $\prod_{P\in C(\kbar)}(g,b)_P$. This can be verified exactly as in the calculation of \cite[Section 3]{Howe}.			
		\end{proof}

			\subsection{Soluble coverings}
					For an isogeny $\phi \colon A \to B$ of semiabelian varieties and $V$ a torsor under $B$, let $\Cov_{\textup{sol}}^\phi(V)$ denote the set of isomorphism classes of $\phi$-coverings $U \to V$ with $U(k) \ne \emptyset$. When $k$ is a global field, let $\Sel^\phi(V)$ denote the set of isomorphism classes of $\phi$-coverings of $V$ that are soluble everywhere locally. Define similarly $\Cov_{\textup{sol}}^n(C)$, $\Cov_{\textup{sol}}^\varphi(C)$, $\Sel^n(C)$ and $\Sel^\varphi(C)$.
					
					Recall that for a nice curve $C$ over $k$ with modulus setup $(n,\frak{m})$, $J(k)_\bullet$ denotes the kernel of the composition $\Upsilon \circ d : J(k) \to \HH^1(J[n]) \to \Br(k)$.
					
			\begin{Lemma}\label{lem:CovsolJ}
				$\Cov_\frak{m}^n(J) \cap \Cov_\textup{sol}^n(J) = d(J(k)_\bullet)$.
			\end{Lemma}
			
			\begin{proof}
				$\Cov_\textup{sol}^n(J) = d(J(k))$ and $\Cov_\frak{m}^n(J) = \ker(\Upsilon)$ by Lemma~\ref{lem:defm}.
			\end{proof}	
			
			The reciproicty law in the Brauer group yields the following.
			
			\begin{Corollary}\label{cor:selJ}
				If $k$ is a global field and $J(k_v)_\bullet = J(k_v)$ for all but at most one prime $v$, then $\Sel^n(J) \subset \Cov_\frak{m}^n(J)$.
			\end{Corollary}
			
			This corollary shows that the subgroup $\Cov_\frak{m}^n(J) \subset \Cov^n(J)$ is large enough to be useful for arithmetic applications. We will derive analogous results for $\Cov_\frak{m}^n(C)$ and $\Cov_\frak{m}^n(J^i)$ as corollaries to the following theorem.
			
%
			
			\begin{Theorem}\label{thm:solublecoverings}
				The group $\HH^1(J[n])$ acts on the sets $\Cov^n(J^i)$ by twisting. This gives rise to simply transitive actions of:
				\begin{enumerate}
					\item\label{p1} $\HH^1(J[n])$ on $\Cov^n(J^i)$, when $[J^i]$ is divisible by $n$ in $\HH^1(J)$;
					\item\label{p2} $\HH^1(J[n])$ on $\Cov^n(C)$, when $[J^i]$ is divisible by $n$ in $\HH^1(J)$;
					\item\label{p3} $\ker(\Upsilon)$ on $\Cov_{\frak{m}}^n(J^i)$, when $[J^i_\frak{m}] \in \varphi_*(\HH^1(A_\frak{m})) \subset \HH^1(J_\frak{m})$;
					\item\label{p4} $\ker(\Upsilon)$ on $\Cov_{\frak{m}}^n(C)$, when $[J^1_\frak{m}] \in \varphi_*(\HH^1(A_\frak{m})) \subset \HH^1(J_\frak{m})$;
					\item\label{p5} $J(k)/nJ(k)$ on $\Cov^n_{\textup{sol}}(J^i)$, when $J^i(k) \ne \emptyset$;
					\item\label{p6} $d(J(k)_\bullet)$ on $\Cov^n_{\textup{sol}}(J^i) \cap \Cov_{\frak{m}}^n(J^i)$, when $J_\frak{m}^i(k) \ne \emptyset$;
				\end{enumerate}
					and, assuming $k$ is a global field, of
				\begin{enumerate}[resume]
					\item\label{p7} $\Sel^n(J)$ on $\Sel^n(J^i)$, when $[J^i] \in n\Sha(J)$;
					\item\label{p8} $\Sel^n(J)$ on $\Sel^n(J^i) \cap \Cov_{\frak{m}}^n(J^i)$, when $[J^i] \in n\Sha(J)$ and for all but at most one prime $v$ of $k$ we have $J(k_v)_\bullet = J(k_v)$ and $J^i_\frak{m}(k_v) \ne \emptyset$.
				\end{enumerate}
			\end{Theorem}

			\begin{proof}
\,\hfill
				\begin{enumerate}
					\item First note that $n$-coverings are $J[n]$-torsors. As in \cite[Section 2.2]{Skorobogatov}, the low degree terms of the Hochschild-Serre spectral sequence give an exact sequence
					\begin{equation*}\label{eq:HS}
						0 \to \HH^1(\Gal_k,J[n]) \to \HH^1_{\textup{\'et}}(J^i,J[n]) \to \HH^0(\Gal_k,\HH^1_{\textup{\'et}}(J^i_\kbar,J[n])) \stackrel{\partial}\to \HH^2(\Gal_k,J[n])\,.
					\end{equation*}
					There exists an $n$-covering of $J^i_\kbar$ and the image of its class under $\partial$ is the obstruction to the existence of an $n$-covering of $J^i$. This obstruction coincides with the coboundary of $[J^i]$ arising from the exact sequence $0 \to J[n] \to J \to J \to 0$ (see \cite[Lemma 2.4.5]{Skorobogatov}). In particular, if $[J^i]$ is divisible by $n$, then $\Cov^n(J^i) \ne \emptyset$. In this case $\HH^1(J[n])$ acts simply transitively on $\Cov^n(J^i)$ by exactness of the sequence above.
					\item It follows from geometric class field theory that the map $\Cov^n(J^1) \to \Cov^n(C)$ given by pullback is a bijection which respects the action of $\HH^1(J[n])$, so \reff{p1} $\Rightarrow$ \reff{p2}.
					\item 
					As in the proof of \reff{p1}, the condition $[J^i_\frak{m}] \in \varphi_*(\HH^1(A_\frak{m})) \subset \HH^1(J_\frak{m})$ ensures that $\Cov^\varphi(J^i_\frak{m})$ is nonempty, and hence is a principal homogeneous space for $\HH^1(A_\frak{m}[\varphi])$.
					The map $q : \Cov^\varphi(J^i_\frak{m}) \to \Cov^n(J^i)$ is a map of principal homogeneous spaces, compatible with the homomorphism $s: \HH^1(A_\frak{m}[\varphi]) \to \HH^1(J[n])$ coming from the cohomology of the exact sequence $1 \to T'[\varphi] \to A_\frak{m}[\varphi] \to J[n] \to 0$. The image of $q$ is $\Cov_\frak{m}^n(J^i)$, while the image of $s$ is $\ker(\Upsilon)$.

					\item This follows from \reff{p3} by pullback.
					\item If $J^i(k) \ne \emptyset$, then $\Cov^n_{\textup{sol}}(J^i)\ne \emptyset$ (since in this case $[J^i] = 0$ in $\HH^1(J)$ is divisible by $n$). The difference of any two soluble $n$-coverings has trivial image in $\HH^1(J)$, hence must lie in the image of the Kummer map $d:J(k)/nJ(k) \hookrightarrow \HH^1(J[n])$.
					\item By assumption $J^i_\frak{m}(k) \ne \emptyset$, so $\Cov_\textup{sol}^\varphi(J^i_\frak{m}) \ne \emptyset$. Then $q(\Cov_\textup{sol}^\varphi(J^i_\frak{m})) \subset \Cov_\frak{m}^n(J^i) \cap \Cov_\textup{sol}^n(J)$ is nonempty. The result now follows from \reff{p3} and \reff{p5} since $d(J(k)_\bullet) = d(J(k)/nJ(k)) \cap \ker(\Upsilon)$.

					\item Since $[J^i] \in n\Sha(J)$, we have that $\Sel^n(J^i) \ne \emptyset$. One then argues as in \reff{p5} (everywhere locally) to see that the difference of two locally soluble $n$-coverings of $J^i$ gives an element of $\Sel^n(J)$.
					\item First we claim that $\Cov^\varphi(J^i_\frak{m})$ and, hence, $\Cov_\frak{m}^n(J^i)$ are nonempty. As in the proof of \reff{p1}, there exists a $\varphi$-covering of $(J^i_\frak{m})_{\kbar}$ and the obstruction to Galois descent is an element $o \in \HH^2(A_\frak{m}[\varphi])$. There is an exact sequence $\Br(k)[n] = \HH^2(T'[\varphi]) \to \HH^2(A_\frak{m}[\varphi]) \to \HH^2(J[n])$ and the image of $o$ in $\HH^2(J[n])$ is the obstruction to the existence of an $n$-covering of $J^i$. We have assumed $[J^i] \in n\Sha(J)$, so $o$ is the image of an element from $\Br(k)$. However, $o$ must vanish everywhere locally since we have assumed $J^i_\frak{m}$ is locally soluble. So $o$ is trivial by the local-global principle for the Brauer group.

					Now suppose $(F,\pi) \in \Sel^n(J^i)$. By \reff{p1} there exists some $\xi \in \HH^1(J[n])$ such that the twist $\xi \cdot (F,\pi)$ lies in $\Cov_\frak{m}^n(J^i)$. We will show that $\xi \in \ker(\Upsilon)$ which, in light of \reff{p3}, shows that $(F,\pi) \in \Cov_\frak{m}^n(J^i)$. Thus, $\Sel^n(J^i) \subset \Cov_\frak{m}^n(J^i)$ and the conclusion of \reff{p8} follows from \reff{p7}.

					Let $v$ be a prime such that $J(k_v)_\bullet = J(k_v)$. Together \reff{p5} and \reff{p6} show that $\Cov_\textup{sol}^n(J^i_{k_v}) \subset \Cov_\frak{m}(J^i_{k_v})$. Since $\res_v(F,\pi) \in \Cov_\textup{sol}^n(J^i_{k_v})$, \reff{p3} implies that $\res_v(\xi) \in \ker(\Upsilon)$. Since $J(k_v)_\bullet = J(k_v)$ for all but at most one prime, the reciprocity law in the Brauer group gives that $\xi \in \ker(\Upsilon)$.
				\end{enumerate}
	
			\end{proof}

			\begin{Corollary}\label{cor:selC}
				$\Cov_\textup{sol}^n(C) \subset \Cov^n_\frak{m}(C)$ and if $k$ is a global field, then $\Sel^n(C) \subset \Cov_{\frak{m}}^n(C)$.
			\end{Corollary}

			\begin{proof}
				For the first statement we may assume $\Cov^n_\textup{sol}(C) \ne \emptyset$. Then $C(k) \ne \emptyset$ and, hence, $J^1_\frak{m}(k), J^1(k) \ne \emptyset$ and $J(k)_\bullet = J(k)$. So \reff{p5} and \reff{p6} show that $\Cov_\textup{sol}^n(J^1) \subset \Cov_\frak{m}^n(J^1)$. Then $\Cov_\textup{sol}^n(C) \subset p(\Cov^n_\textup{sol}(J^1)) \subset p(\Cov^n_\frak{m}(J^1)) = \Cov^n_\frak{m}(C)$. To prove the second statement we may assume $\Sel^n(C)$ is nonempty. Then the hypothesis of Theorem~\ref{thm:solublecoverings}\reff{p8} in case $i=1$ is satisfied, so this together with Theorem \reff{thm:solublecoverings}\reff{p7} shows that $\Sel^n(J^1) \subset \Cov_\frak{m}^n(J^1)$. The result now follows by applying the pullback map as in the proof of the first statement.
			\end{proof}

			\begin{Corollary}\label{cor:Sel}
				If $J(k)_\bullet = J(k)$ and $\Div^i(C) \ne \emptyset$, then $\Cov_\textup{sol}^n(J^i) \subset \Cov_\frak{m}^n(J^i)$. If $k$ is a global field and for all but at most one prime $v$ of $k$, $\Div^i(C_{k_v}) \ne \emptyset$ and $J(k_v)_\bullet = J(k_v)$. Then $\Sel^n(J^i) \subset \Cov_{\frak{m}}^n(J^i)$.
			\end{Corollary}
			
			\begin{proof}
				If $\Sel^n(J^i) = \emptyset$ there is nothing to prove. Otherwise, the hypotheses of \reff{p8} is satisfied since our assumption $\Div^i(C_{k_v}) \ne \emptyset$ implies $J^i_\frak{m}(k_v) \ne \emptyset$ by Lemma~\ref{lem:Jipoints}. Together \reff{p7} and \reff{p8} give the result. 
			\end{proof}

			\section{The descent setup}\label{sec:descentmap}
			We recall the following definition from \cite{BPS} (where the defined object is referred to as a \textit{fake descent setup}).
			
			\begin{Definition}\label{def:descentsetup}
				Let $C$ be a nice curve over $k$. A \defi{descent setup for $C$} is a triple $(n,\Delta,\beta)$ consisting of a positive integer $n$ not divisible by the characteristic of $k$, a nonempty finite \'etale $k$-scheme $\Delta = \Spec L$, and a divisor $\beta \in \Div(C\times \Delta)$ such that $n\beta = \frak{m}\times \Delta + \divv(f_\frak{m})$ for some $\frak{m} \in \Div(C)$ and $f_\frak{m} \in k(C\times \Delta)^\times$.
			\end{Definition}
			
			If the divisor $\frak{m}$ appearing in the definition is effective, reduced and base point free, then $(n,\frak{m})$ is a modulus setup, which we say is \defi{associated to $(n,\Delta,\beta)$}. For each $\delta \in \Delta(\kbar)$, $\beta_\delta \in \Div(C_\kbar)$ is a divisor such that $n\beta_\delta - \frak{m}$ is principal. So the class of $\beta_\delta$ in $\calJ$ lies in $\calJ[n]$. This gives rise to a map $\Res_\Delta \Z/n\Z \to \calJ[n]$ sending $\sum_{\delta \in \Delta(\kbar)}c_\delta$ to the class of $\sum c_\delta\beta_\delta$. There is trace map $\Tr:\Res_\Delta\Z/n\Z \to \Z/n\Z$ whose kernel we denote by $\Res^0_\Delta \Z/n\Z$. This fits into a commutative and exact diagram,

			\begin{equation}\label{diag:alphavee}
				\xymatrix{
					0 \ar[r] & \Res^0_\Delta\Z/n\Z \ar[d]\ar[r] & 
					\Res_\Delta\Z/n\Z \ar[d]\ar[r]^-\Tr & 
					\Z/n\Z  \ar@{=}[d] \ar[r] & 0\\ 
					0 \ar[r] & J[n] \ar[r] & \calJ[n] \ar[r]^-{\frac{1}{\ell}\deg} & \Z/n\Z \ar[r] & 0			
				}
			\end{equation}
		
			\begin{Definition}
			We say that $(n,\Delta,\beta)$ is an \defi{$n$-descent setup} if the vertical maps in~\eqref{diag:alphavee} are surjective and the divisors $\beta_\delta \in \Div(C_\kbar)$ are effective and have no common support.
			\end{Definition}

			We note that if $(n,\Delta,\beta)$ is an $n$-descent setup, then the divisor $\frak{m}$ appearing in the definition is base point free, as it is linearly equivalent to each of the $n\beta_\delta$, which by assumption have no common support. Thus $(n,\frak{m})$ is a modulus setup.

			The following examples show that all of the modulus setups considered in Section~\ref{sec:modulus} are associated to an $n$-descent setup. Details for \ref{ex:hyp2d} and \ref{ex:canonical2d} may be found in \cite[Examples 6.9]{BPS}, while \ref{ex:g1nd} is considered in \cite{CreutzMathComp}.
			
			\begin{enumerate}[label=\textbf{D.\arabic*},ref=Example D.\arabic*]
					\item\label{ex:hyp2d} Suppose $C$ is a double cover of $\PP^1$ which is not ramified over $\infty$. Let $\Delta(\kbar)$ be the set of ramification points and take $\beta$ to be the diagonal embedding of $\Delta$ in $C \times \Delta$. Then $(2,\Delta,\beta)$ is a $2$-descent setup. Taking $\frak{m}$ be the pullback of $\infty \in \Div(\PP^1)$ we recover the modulus setup in \ref{ex:M_doublecover}.
					\item\label{ex:g1nd} Suppose $C$ is a genus $1$ curve of degree $n$ in $\PP^{n-1}$ (or equivalently, a genus $1$ curve together with the linear equivalence class of a $k$-rational divisor of degree $n$). We obtain an $n$-descent setup by taking $\Delta$ to be the set of $n^2$ \defi{flex points} (i.e. points $x \in C(\kbar)$ such that $n.x$ is a hyperplance section) and $\beta$ to be the diagonal embedding of $\Delta$ in $C \times \Delta$. Taking $\frak{m}$ to be a generic hyperplane section recovers the modulus setup in \ref{ex:M_g1}.
					\item\label{ex:canonical2d} Suppose $C$ is a nonhyperelliptic curve of genus $\ge 2$. We obtain a $2$-descent setup for $C$ by taking $\Delta$ to be the $\Gal_k$-set of odd theta characteristics. A theta characteristic is a line bundle $\theta$ on $C$ whose square is the canonical bundle. By definition $\theta$ is odd if $h^0(X,\theta) \equiv 1 \bmod 2$, which implies in particular that $\theta$ may be represented by an effective divisor. By \cite[Proposition 5.8]{BPS} there is some effective $\beta \in \Div(C\times \Delta)$ such that $[\beta_\delta] = \delta$ for $\delta \in \Delta(\kbar)$. We can take $\frak{m}$ to be an effective canonical divisor and thus recover the modulus setup in \ref{ex:M_canonical}.	

				\end{enumerate}

\subsection{Descent maps}

		Let $C$ be a nice curve over $k$ with an $n$-descent setup $(n,\Delta,\beta)$ and associated modulus setup $(n,\frak{m})$. Let $L$ denote the \'etale algebra corresponding to $\Delta$, i.e., $\Delta = \Spec L$. Call a divisor on $C$ \defi{good} if its support is disjoint from $\frak{m}$ and all $\beta_\delta$.

		\begin{Lemma}\label{lem:deffm}
			Let $f_\frak{m} \in k(C\times\Delta)^\times$ be as in Definition~\ref{def:descentsetup}. Evaluation of $f_\frak{m}$ at good divisors induces homomorphisms
			\[
				f_{\frak{m}} : \Pic_{C_\frak{m}}(k) \to L^\times/L^{\times n}\,,\quad\text{and} \quad f_\frak{m} : \Pic(C) \To L^\times/k^\times L^{\times n}\,.
			\]
		\end{Lemma}

		\begin{proof}
			By Lemma~\ref{lem:Jipoints} and the moving lemma, all elements of $\Pic_{C_\frak{m}}(k)$ can be represented by a good $k$-rational divisor and $\Pic(C)$ is the image of $\Pic_{C_\frak{m}}(k)$. Suppose $D$ is a good $k$-rational divisor and $D = \divv(g)$ for some $g \in k(C)^\times$. By Weil reciprocity $f_\frak{m}(D) = g(\divv(f_\frak{m})) = g(n\beta-\frak{m}\times\Delta) = N(g|_\frak{m})^{-1}g(\beta)^n \in k^\times L^{\times n}$. This shows that the second map is well defined. If the class of $D$ is trivial in $\Pic_{C_\frak{m}}(k)$, then there is such a $g$ with $g|_\frak{m} = 1$ (cf. Lemma~\ref{lem:2.5}), so that $f_\frak{m}(D) \in L^{\times n}$.
		\end{proof}

	Dualizing~\eqref{diag:alphavee} and taking Galois cohomology yields a commutative and exact diagram,
		\begin{equation}\label{eq:descentsetup}
			\xymatrix{
				\HH^1(T'[\varphi]) \ar[r]\ar@{=}[d] 
				& \HH^1(A_\frak{m}[\varphi]) \ar[r]\ar[d]^{\alpha_\frak{m}}
				& \HH^1(J[n]) \ar[d]^-\alpha\ar[r]^\Upsilon
				& \HH^2(T'[\varphi])\ar@{=}[d]\\
				k^\times/k^{\times n} \ar[r]
				& L^\times/L^{\times n} \ar[r]
				& \HH^1\left(\frac{\Res_\Delta\mu_n}{\mu_n}\right)\ar[r]
				& \Br(k)[n]
			}
		\end{equation}

		This is related to the maps in Lemma~\ref{lem:deffm} and $\varphi$-coverings as follows.

		\begin{Theorem}\label{thm:descentmaps}
			For $i \in \{0,1\}$ there is an $\alpha_\frak{m}$-equivariant map $\alpha^i_\frak{m} : \Cov^\varphi(J^i_\frak{m}) \to L^\times/L^{\times n}$, functorial in $k$ and such that for any $(F_\frak{m},\pi) \in \Cov^\varphi(J^i_\frak{m})$ and $P \in \pi(F_\frak{m}(k))$ we have $\alpha_\frak{m}(F_\frak{m},\pi) = f_\frak{m}(P)$. 
		\end{Theorem}
		
		We prove the cases $i=0$ and $i=1$ separately below, after making some remarks and stating two corollaries that will be used in the following section. Proofs of the corollaries follow the proof of the theorem.	
		
		\begin{Remark}\label{rem:partition}
			The set of rational points on $J^i_\frak{m}$ may be partitioned as
			\[J^i_\frak{m}(k) = \coprod_{(F_\frak{m},\pi) \in \Cov^\varphi(J^i_\frak{m})} \pi(F_\frak{m}(k))\,.\]
			The theorem says that the map $f_\frak{m} : \Pic^i_{C_\frak{m}}(k) = J^i_\frak{m}(k) \to L^\times/L^{\times n}$ is constant on each of the sets appearing in this partition and that the value on each is equal to the image of the corresponding covering under $\alpha^i_\frak{m}$.			
		\end{Remark}

		\begin{Corollary}\label{cor:alphais}
			For $i \in \{0,1\}$ there is an $\alpha$-equivariant map $\alpha^i : \Cov^n_\frak{m}(J^i) \to L^\times/k^\times L^{\times n}$, functorial in $k$ and such that for any $(F,\pi) \in \Cov^n_\frak{m}(J^i)$ and $P \in \pi(F(k)) \cap \Pic^i(C)$ we have $\alpha^i(F,\pi) = f_\frak{m}(P)$. 
		\end{Corollary}

		\begin{Corollary}\label{cor:alphaC}
			There is an $\alpha$-equivariant map $\alpha^1 : \Cov^n_\frak{m}(C) \to L^\times/k^\times L^{\times n}$, functorial in $k$ and such that for any $(X,\pi) \in \Cov^n_\frak{m}(C)$ and $P \in \pi(X(k))$ we have $\alpha^1(X,\pi) = f_\frak{m}(P)$. 
		\end{Corollary}

		\begin{Remark}
			There are partitions of the sets of rational points
			\[
				C(k) = \coprod_{(X,\pi) \in \Cov^n(C)} \pi(X(k)) \quad\text{and}\quad J^i(k) = \coprod_{(F,\pi) \in \Cov^n(C)} \pi(F(k))
			\]
			By Corollary~\ref{cor:selC}, $\Cov_\textup{sol}^n(C) \subset \Cov_\frak{m}^n(C)$, so Corollary~\ref{cor:alphaC} says that the map $f_\frak{m} : C(k) \to L^\times/k^\times L^{\times n}$ is constant on each of the sets appearing in this partition and that the value on each is equal to the image of the corresponding covering under the descent map. A similar statement holds for $J^i$, provided $\Pic^i(C) = J^i(k)$, in which case $\Cov_\textup{sol}^n(J^i) \subset \Cov_\frak{m}^n(J^i)$ by Lemma~\ref{lem:CovsolJ} and Corollary~\ref{cor:Sel}.
		\end{Remark}
		
		\begin{Remark}
			The subgroup $J(k)_\bullet = \ker(\Upsilon \circ d) \subset J(k)$ is the largest subgroup of $J(k)$ on which one can define a homomorphism $f:J(k)_\bullet \to L^\times/k^\times L^{\times 2}$ such that $f$ agrees with $f_\frak{m}$ on $\Pic^0(C)$ and $f \circ d$ agrees with $\alpha^0 \circ d$ as in Corollary~\ref{cor:alphais}. This follows from a diagram chase in~\eqref{eq:descentsetup}. Corollary~\ref{Cor:Ups} shows that $J(k)_\bullet = J(k)$ when $C$ is a nonhyperelliptic curve defined over a local or global field and has $2$-descent setup as in~\ref{ex:canonical2d}. This is rather surprising given that it is not generally true that $J(k) = \Pic^0(C)$ and, moreover, that $J(k)_\bullet = \Pic^0(C)$ for a hyperelliptic curve with $2$-descent setup as in~\ref{ex:hyp2d} (see \cite[Corollary 10.6]{PoonenSchaefer}). It would be very interesting to determine how this extended map $f$ could be computed explicity over, say, a local field.
		\end{Remark}

		\begin{proof}[Proof of Theorem~\ref{thm:descentmaps} in the case $i = 0$]

			Let $d_\frak{m} : J_\frak{m}(k) \to \HH^1(A_\frak{m}[\varphi])$ denote the connecting homomorphism from the exact sequence $0 \to A_\frak{m}[\varphi] \to A_\frak{m} \to J_\frak{m} \to 0$. Under the identification $\HH^1(A_\frak{m}[\varphi]) = \Cov^\varphi(J_\frak{m})$, the coboundary map $d_\frak{m}$ sends $P \in J_\frak{m}(k)$ to the class of the $\varphi$-covering $A_\frak{m} \to J_\frak{m}$ given by $Q \mapsto \varphi(Q) + P$. So the following lemma proves Theorem~\ref{thm:descentmaps} in the case $i = 0$.
		\end{proof}

		\begin{Lemma}\label{lem:abc}
			The composition $J_\frak{m}(k) \stackrel{d_\frak{m}}\To \HH^1(A_\frak{m}[\varphi]) \stackrel{\alpha}{\To} L^\times/L^{\times n}$ is equal to $f_\frak{m}$.
		\end{Lemma}

		\begin{proof}
			Let $D \in \Div_\frak{m}(C)$ be a good divisor representing $P \in J_\frak{m}(k)$. Choose a good divisor $E \in \Div_\frak{m}(C_\kbar)$ such that $nE - D = \divv(g)$ for some $g \in \kbar(C_\kbar)^\times$ with $g|_\frak{m} = 1$. This is possible since $J_\frak{m}(\kbar)$ is a divisible group. Then $d_\frak{m}(P)$ is represented by the $1$-cocycle $\xi_\sigma = [{}^\sigma E - E] \in A_\frak{m}[\varphi]$. Note that $\divv({}^\sigma g/g) = n({}^\sigma E - E)$. The image of $\xi$ under $\alpha:\HH^1(A_\frak{m}[\varphi]) \to \HH^1(\Res_\Delta\mu_n)$ is represented by $e(\xi_\sigma,\beta)$, where $e$ is the pairing defined in Proposition~\ref{prop:WP2}. From the definition of the $e$ pairing we have
			\[
				e(\xi_\sigma,\beta) = f_\frak{m}({}^\sigma E-E)/({}^\sigma g/g)(\beta) = {}^\sigma b/b\,,
			\]
			where $b = f_\frak{m}(E)/g(\beta)$. Thus, the image of $\alpha(\xi)$ under $\HH^1(\Res_\Delta\mu_n) \simeq L^\times/L^{\times n}$ is represented by $b^n = f_\frak{m}(nE)/g(n\beta) = f_\frak{m}(D + \divv(g))/g(\frak{m}\times\Delta+ \divv(f_\frak{m})) = f_\frak{m}(D)/g(\frak{m}\times\Delta) = f_\frak{m}(D)$, where the last two equalities follow from Weil reciprocity and the fact that $g|_\frak{m} = 1$, respectively.
		\end{proof}

		\begin{proof}[Proof of Theorem~\ref{thm:descentmaps} in the case $i = 1$]
		Given $(F_\frak{m},\rho) \in \Cov^\varphi(J^1_\frak{m})$, let $(Y,\pi) \in \Cov^\varphi(C)$ be its image under the pullback map. As in the proof of Proposition~\ref{prop:cft} the extension $\kbar(Y_\kbar)$ contains $n$-th roots $g_\delta$ of $f_{\frak{m},\delta}$ for each $\delta \in \Delta(\kbar)$. Evidently $\divv({}^\sigma(g_\delta)) = \divv(g_{\sigma(\delta)})$ for any $\sigma \in \Gal_k$, so by Hilbert's Theorem 90 there is a function $h \in k(Y\times \Delta)^\times$ such that $\divv(g_\delta) = \divv(h_\delta)$. Then $\pi^*f_\frak{m}/h^n \in k(Y\times \Delta)^\times$ has trivial divisor, so must equal some constant function $c \in L^\times = k(\Delta)^\times$. The class of $c$ in $L^\times/ L^{\times n}$ is independent of the choice for $h$. Thus we have a well defined map $\alpha^1:\Cov^\varphi(J^1_\frak{m}) \to L^\times/L^{\times n}$ sending $(F_\frak{m},\rho)$ to the class of $c$.

Now suppose $D \in \Div^1(C)$ is a good divisor. For each $P \in C(\kbar)$ in the support of $D$ choose some $P' \in Y(\kbar)$ such that $\pi(P') = P$ and set $D' = \sum_P \ord_P(D)P' \in \Div(\Ybar)$. Let $\sigma \in \Gal_k$. Since $D$ is $k$-rational we can write ${}^\sigma(D') = \sum_P \ord_P(D)P'_\sigma$ with $\pi(P'_\sigma) = P$. The restriction of $\pi$ to the open subscheme $Y_0 = Y - \pi^{-1}(\frak{m})$ is an $A_\frak{m}[\varphi]$-torsor over $C_0 = C - \frak{m}$. Thus, for each $P \in \textup{Supp}(D)$ and $\sigma \in \Gal_k$ there is a unique $\gamma_{P,\sigma} \in A_\frak{m}[\varphi](\kbar)$ such that $\gamma_{P,\sigma}\cdot P' = P'_\sigma$. Set $\gamma_\sigma = \sum_P \ord_P(D)\gamma_{P,\sigma}$, which we interpret as a $1$-cocycle taking values in $A_\frak{m}[\varphi]$. Since $\pi(D') = D$, $\gamma_\sigma$ represents the class in $\HH^1(A_\frak{m}[\varphi])$ of the torsor $\rho^{-1}([D]) \subset F_\frak{m}$.

	From the relation defining $c$ we have $f_\frak{m}(D)/c = f_\frak{m}\circ \pi(D')/c = h(D')^n$. This represents a class in $\HH^1(\Res_\Delta\mu_n)$ given by the $1$-cocyle
	\[
		\eta_\sigma = {}^\sigma (h(D'))/h(D') = h({}^\sigma D')/h(D') = \prod_P [h(\gamma_{P,\sigma}\cdot P')/h(P')]^{\ord_P(D)}\,.
	\]
	By Lemma~\ref{lem:WP} this can be expressed in terms of the extended Weil pairing of Proposition~\ref{prop:WP2} as
	\[
		\eta_\sigma = \prod_P e(\gamma_{P,\sigma},\beta)^{\ord_P(D)} = e\left(\gamma_\sigma,\beta\right)\,.
	\] 	
	This also represents the image of $\gamma_\sigma$ under the map $\alpha : \HH^1(A_\frak{m}[\varphi]) \to \HH^1(\Res_\Delta\mu_n)$. Thus, $f_\frak{m}(D)/c = f_\frak{m}(D)/\alpha^1(F_\frak{m},\rho)$ is equal to $\alpha(\rho^{-1}([D]))$. In particular, $f_\frak{m}(D) = \alpha^1(F_\frak{m},\rho)$ whenever the fiber $\rho^{-1}([D])$ contains a $k$-point. This is the property stated in the theorem.

	Let us show that $\alpha^1_\frak{m}$ is $\alpha_\frak{m}$-equivariant. Suppose $(Y_\xi,\pi_\xi)$ is the twist of $(Y,\pi)$ by $\xi \in \HH^1(k,A_\frak{m}[\varphi])$. By definition there is an isomorphism $\psi :  (Y_\xi)_\kbar \to Y_\kbar$ such that $\pi \circ \psi = \pi_\xi$ and ${ }^{\sigma}\psi(x) = \xi_\sigma\cdot \psi(x)$ for any $\sigma \in \Gal_k$, where $\xi_\sigma \in A_\frak{m}[\varphi] \simeq \Aut(Y_\kbar/C_\kbar)$. Let $h,c$ and $h_\xi,c_\xi$ be as in the definition of $\alpha_\frak{m}^1$. We must show that $c_\xi/c$ and $\alpha(\xi)$ give the same class in $L^\times/L^{\times n} \simeq \HH^1(\Res_\Delta\mu_n)$. We have $c_\xi h_\xi^n= \pi_\xi^*f_\frak{m} = (\pi \circ \psi)^*f_\frak{m} = c (h \circ \psi)^n\,.$ Thus, $c_\xi/c = (h(\psi(Q))/h_\xi(Q))^n$, for any $Q \in Y_\xi(\kbar)$ where this expression is defined and nonzero. So the class of $c_\xi/c$ in $L^\times/L^{\times n} \simeq \HH^1(\Res_\Delta \mu_n)$ is represented by the $1$-cocycle
	\begin{align*}
		\nu_\sigma 
		&= {}^\sigma\left(\frac{h(\psi(Q))}{h_\xi(Q)}\right)\left(\frac{h_\xi(Q)}{h(\psi(Q))}\right)
		= \left(\frac{h(\xi_\sigma\cdot\psi({}^\sigma Q))}{h(\psi(Q))}\right)\left(\frac{h_\xi(Q)}{h_\xi({}^\sigma Q)}\right)\\
		&= \left(\frac{h(\xi_\sigma\cdot\psi({}^\sigma Q))}{h(\psi({}^\sigma Q))}\right)\underbrace{\left(\frac{h(\psi({}^\sigma Q))}{h_\xi({}^\sigma Q)}\right)}_{c_\xi/c}\underbrace{\left(\frac{h_\xi(Q)}{{h(\psi(Q))}}\right)}_{c/c_\xi} = \left(\frac{h(\xi_\sigma\cdot\psi({}^\sigma Q))}{h(\psi({}^\sigma Q))}\right)\\
		&= e(\xi_\sigma,\beta) = \alpha(\xi_\sigma)\,,
	\end{align*}
	where the final line follows from Lemma~\ref{lem:WP} as above.
\end{proof}

		\begin{proof}[Proof of Corollary~\ref{cor:alphais}]
			If two elements of $\Cov^\varphi(J^i_\frak{m})$ have the same image in $\Cov_\frak{m}^n(J^i)$, then their images under $\alpha_\frak{m}^i$ differ by an element in $k^\times/(L^{\times n} \cap k^\times)$. This follows from the exactness of~\eqref{eq:descentsetup} and $\alpha_\frak{m}$-equivariance of $\alpha^i_\frak{m}$. Thus there is a unique map $\alpha^i$ fitting into the commutative diagram
			\[
		\xymatrix{
			J_\frak{m}^i(k) \ar[r]^{d_\frak{m}} \ar[d]^s & \Cov^\varphi(J^i_\frak{m}) \ar[r]^{\alpha_\frak{m}^i}\ar[d]^q & L^\times/L^{\times n} \ar[d] \\
			\Pic^i(C) \ar[r]^d & \Cov^n_\frak{m}(J^i) \ar[r]^{\alpha^i} & L^\times/k^\times L^{\times n}
		}
		\]
		Here the maps $d_\frak{m}$ and $d$ are defined by $d_\frak{m}(P) = [A_\frak{m} \ni Q \mapsto \varphi(Q) + P \in J^i]$ and $d(P) = [ J \ni Q \mapsto nQ + P \in J^i ]$. Note that in the case $i = 0$ these agree with the usual connecting homomorphisms. The map $s$ is induced by the canonical map $\Pic_{C_\frak{m}}(k) \to \Pic_C(k)$, and is surjective by Lemma \ref{lem:Jipoints}. Theorem~\ref{thm:descentmaps} shows that the composition along the top row is the map $f_\frak{m}$ of Lemma~\ref{lem:deffm}. Hence, the same is true of the bottom row. Thus $\alpha^i$ has all of the required properties.
\end{proof}

\begin{proof}[Proof of Corollary~\ref{cor:alphaC}]
	The pullback map $p : \Cov^n_\frak{m}(J^1) \to \Cov^n_\frak{m}(C)$ is a bijection by Lemma \ref{lem:defm}. Define $\alpha^1(X,\pi) = \alpha^1(F,\pi)$ where $(X,\pi)$ is the pullback of $(F,\pi)$. The required properties follow immediately from Corollary~\ref{cor:alphais}.
\end{proof}

		\section{Fake Selmer sets}
		
			When $k$ is a global field with completions $k_v$ the map $f_\frak{m}$ induces a commutative diagram,
			\begin{equation}\label{eq:seldiagram}
				\xymatrix{
					\Pic(C) \ar[rr]^-{f_\frak{m}} \ar[d] && \frac{L^\times}{k^\times L^{\times n}} \ar[d]^{\prod \res_v}\\
					\prod_v \Pic(C_{k_v}) \ar[rr]^-{\prod f_{\frak{m},v}}&& \prod \frac{(L\otimes k_v)^\times}{k_v^\times (L\otimes k_v)^{\times n}}
				}
			\end{equation}
			
			\begin{Definition}\label{def:selfaked}
				Suppose $k$ is a global field.
				For any integer $i$, the \defi{fake Selmer set of $J^i$} is the set
				\begin{equation*}
				\Sel^{f_\frak{m}}_\textup{fake}(J^i) := \left\{ \, l \in L^{\times}/k^\times L^{\times n} \;:\; \res_v(l) \in f_{\frak{m},v}(\Pic^i(C_{k_v}))\text{, for all $v$ }\right\}\,.
				\end{equation*}
				The \defi{fake Selmer set of $C$} is the set 
				\begin{equation*}
				\Sel^{f_\frak{m}}_\textup{fake}(C) := \left\{ \, l \in L^{\times}/k^\times L^{\times n} \;:\; \res_v(l) \in f_{\frak{m},v}(C(k_v))\text{, for all $v$ }\right\}\,.
			\end{equation*}
			\end{Definition}
			
\begin{Theorem}\label{thm:selfakeC}
				Suppose $C$ is defined over a global field. If $\Sel^{f_\frak{m}}_\textup{fake}(C) = \emptyset$, then $\Sel^n(C) = \emptyset$.
			\end{Theorem}

			\begin{proof}
				By Corollary~\ref{cor:selC}, $\Cov_\textup{sol}^n(C_{k_v}) \subset \Cov_\frak{m}^n(C_{k_v})$ for each $v$ and $\Sel^n(C) \subset \Cov_\frak{m}^n(C)$. By Corollary~\ref{cor:alphaC} we have $f_\frak{m}(C(k_v)) = \alpha^1(\Cov_\textup{sol}^n(C_{k_v}))$. Thus $\alpha^1(\Sel^n(C)) \subset \Sel^{f_\frak{m}}_\textup{fake}(C)$. 
			\end{proof}

 			\begin{Theorem}\label{thm:selfakeJ1}
				Suppose $C$ is defined over a global field and $\Div^1(C_{k_v}) \ne \emptyset$ for all primes $v$ of $k$. If $\Sel^{f_\frak{m}}_\textup{fake}(J^1) = \emptyset$, then $\Sel^n(J^1) = \emptyset$.
			\end{Theorem}

			\begin{proof}
				As noted in the proof of Lemma~\ref{lem:upsiloniota} we have $\Theta_C(x) = \langle x, [J^1] \rangle$. So the assumption on $\Div^1(C_{k_v})$ implies that $\Pic^0(C_{k_v}) = J(k_v)_\bullet = J(k_v)$.  Thus the hypothesis of Corollary~\ref{cor:Sel} is satisfied and so $\Sel^n(J^1) \subset \Cov_\frak{m}^n(J^1)$. The property of $\alpha^1$ given in Theorem~\ref{thm:descentmaps} together with Corollary~\ref{cor:Sel} gives that $f_{\frak{m},v}(\Pic^1(C_{k_v})) = \alpha^1(\Cov_\textup{sol}^n(J^1_{k_v}))$. It follows that  $\alpha^1(\Sel^n(J)) \subset \Sel^{f_\frak{m}}_\textup{fake}(J^1)$, which gives the result.
			\end{proof}
			
			\begin{Remark}
				The conclusion of the theorem implies that $J^1$ represents a nontrivial element in $\Sha(J)/n\Sha(J)$, not just in $\Sha(J)$. Together with well known properties of the Cassels-Tate pairing this allows one to deduce better lower bounds for $\Sha(J)$ and hence better upper bounds for the rank of $J(k)$. This is illustrated in the example in Section~\ref{sec:Examples}.
			\end{Remark}
			
			\begin{Remark}				  
				When $C$ has a descent setup as in \ref{ex:hyp2d} and \ref{ex:g1nd} a proof of Theorem~\ref{thm:selfakeJ1} can be found in \cite[Prop. 5.4]{CreutzANTSX} and \cite[Theorem 5.2]{CreutzMathComp}, respectively.	
			  \end{Remark}

			\subsection{Descent on $J$}
The results of~\cite[Section 10]{BPS} show that from knowledge of $\Sel^{f_\frak{m}}_\textup{fake}(J)$ one can often determine $\Sel^n(J)$. For this to work one must at least have that $\alpha(\Sel^n(J))$ is contained in the image of $L^\times/k^\times L^{\times n}$ (cf.~\eqref{eq:descentsetup}) or, equivalently, $\Sel^n(J) \subset \ker(\Upsilon) = \Cov_\frak{m}^n(J)$. This can be ensured by imposing hypotheses on $C$ such as \cite[Hypothesis 10.1]{BPS} that the map $\Pic^0(C) \to J(k)/nJ(k)$ is surjective both globally and locally. The results of Section $2$ allow us to extract information concerning $\Sel^n(J)$ in a number of cases where \cite[Hypothesis 10.1]{BPS} does not hold.

\begin{Theorem}\label{thm:descentJ}
	Suppose $C$ is defined over a global field $k$ and $J^2(k) \ne \emptyset$. Let $N$ be the number of primes $v$ such that $\coker(\Pic^0(C_{k_v}) \to J(k_v)/2J(k_v))\ne 0$. Suppose that either of the following holds
	\begin{enumerate}
		\item $C$ is a nonhyperelliptic curve with a $2$-descent setup as in~\ref{ex:canonical2d}, or
		\item $C$ is a hyperelliptic curve with a $2$-descent setup as in~\ref{ex:hyp2d} and $N \le 1$.
	\end{enumerate}
	Then 
	\[
		\dim_{\F_2}(\alpha(\Sel^2(J))) \le \dim_{\F_2}\Sel^{f_\frak{m}}_\textup{fake}(J) + \max\{ 0, N - 1 \}\,.
	\]
\end{Theorem}

\begin{Remark}
	The kernel of $\alpha : \HH^1(J[2]) \to \HH^1(\Res_\Delta\mu_2/\mu_2)$ can be computed from the Galois action on $\Delta$, thus allowing us to extract upper bounds for $\dim_{\F_2}(\Sel^2(J))$ as well.
\end{Remark}

	\begin{proof}
		For each prime $v$, let $M_v := \coker(\Pic^0(C_{k_v}) \to J(k_v)/2J(k_v))$ and let $T$ be the (finite) set of primes where $M_v$ is nontrivial. In the nonhyperelliptic case we have $\ell = g-1$ by Corollary~\ref{Cor:Ups}, so $J(k_v)_\bullet = J(k_v)$ for all primes $v$. In the hyperelliptic case the assumption $N \le 1$ implies $J(k_v)_\bullet = J(k_v)$ fails for at most one prime $v$. In both cases $\Sel^n(J) \subset \Cov_\frak{m}^n(J)$ by Corollary~\ref{cor:selJ}.
		
		For $v \not\in T$ we have $\Pic^0(C_{k_v}) = J(k_v)_\bullet = J(k_v)$ and $f_\frak{m}(\Pic^0(C_{k_v})) = \alpha^0(\Cov^n_\textup{sol}(J_{k_v}))$ by Lemma~\ref{lem:CovsolJ} and Theorem~\ref{cor:alphais}. So if $T = \emptyset$, then we have $\alpha(\Sel^n(J)) \subset \Sel^{f_\frak{m}}_\textup{fake}(J)$ and the result holds.

		Let use assume $N = \# T > 0$. Let $K_v := \alpha(d(J(k_v)))$, $\Lambda_v := f_\frak{m}(\Pic^0(C_{k_v}))$. Identifying $J(k_v)/2J(k_v)$ with its image under $d$ and using Lemma~\ref{lem:abc} we obtain a commutative diagram of $\F_2$-linear maps
		\[
			\xymatrix{
				\Sel^2(J) \ar[r]\ar[d]^\alpha & \bigoplus_{v \in T} J(k_v)/2J(k_v) \ar[r]\ar[d]^\alpha & \bigoplus_{v \in T}M_v \ar[d]^\alpha \\
				\alpha(\Sel^2(J)) \ar[r] & \bigoplus_{v \in T}K_v \ar[r]& \bigoplus_{v \in T} K_v/\Lambda_v 
			}
		\]
		The maps $\Theta_{C_{k_v}}$ of~\eqref{eq:defThetaC} induce an isomorphism $\oplus M_v \to \oplus \Br(k_v)[2] \simeq \F_2^N$. Since $J^2(k) \ne \emptyset$, $[J^1] \in \HH^1(J)[2]$. Therefore, there is a lift $\eta$ of $[J^1]$ to $\HH^1(J[2])$. Given $\xi \in \Sel^2(J)$ let $b = \xi \cup_e \eta \in \Br(k)$ and for $v \in T$ let $x_v \in J(k_v)$ be such that $d(x_v) = \res_v(\xi)$. Compatibility of the Tate pairing with the Weil pairing cup product (as noted in the proof of Lemma~\ref{lem:upsiloniota}) gives $\res_v(b) = \langle d(x_v),[(J^1)_{k_v}]\rangle = \Theta_{C_{k_v}}(x_v)$. Therefore global reciprocity in $\Br(k)$ implies that the image of $\Sel^2(J)$ in $\bigoplus M_v \simeq \F_2^N$ is contained in a hyperplane. Since the vertical map on the right is surjective, this shows that the rank of the composition along the bottom row of the diagram is at most $N-1$. On the other hand, the kernel is $\Sel^{f_\frak{m}}_\textup{fake}(J)$.
	\end{proof}

		Here is an instance where we can prove that $J(k_v) \ne \Pic^0(C_{k_v})$.
		\begin{Lemma}\label{lem:instance}
			Suppose that $C$ is a curve of genus $g$ and either
			\begin{enumerate}
				\item $k$ is a local field such that $J^1(k) = \emptyset$, or
				\item $C$ is defined over a global field $K$ and $k = K_v$ is the unique completion of $K$ such that $\Pic^1(C_k) = \emptyset$. Assume further that $\Div^{g-1}(C_{K_v}) \ne \emptyset$ for all primes $v$.
			\end{enumerate}
			Then $\coker(\Pic^0(C_k) \to J(k)/2J(k)) \ne 0$.
		\end{Lemma}
		
		\begin{proof}
			The map $\Theta_C$ is related to the Tate pairing by the rule $\Theta_C(x) = \langle x, [J^1]\rangle$. The assumption in $(1)$ is that $[J^1]$ is nontrivial in $\HH^1(J)$, so the result follows from nondegeneracy of the Tate pairing. In case (2), the second assumption implies that the Cassels-Tate pairing is alternating by \cite[Corollary 11]{PoonenStoll}. If $J^1(k) \ne \emptyset$, then $J^1 \in \Sha(J)$ and \cite[Theorem 11]{PoonenStoll} shows that $J^1$ pairs nontrivially with itself, a contradiction. Hence the hypothesis of (1) is satisfied.
		\end{proof}

\section{Examples}\label{sec:examples}

	Computations in this section were performed with the Magma Computer Algebra System described in \cite{magma}.

	\subsection{Example of explicit descent on $J^1$}\label{sec:Examples}
		 	
		 		\begin{Theorem}\label{thm:example}
		 			Let $C$ denote the genus $3$ curve in $\PP^2_\Q$ given by the vanishing of
		 			\[
		 				x^4 + 5x^3y + 9x^3z + 9x^2y^2 + 9x^2yz + xy^3 - 8xy^2z - 8xz^3 - 
    6y^4 - 3y^3z - 8y^2z^2 - 2yz^3 - 3z^4\,
		 			\]
		 			and let $J$ be the Jacobian of $C$. Then, assuming the generalized Riemann hypothesis, $J(\Q) \simeq \Z$ and $\Sha(J)[2^\infty] \simeq \Z/2\Z\times \Z/2\Z$. Furthermore, the curve $C$ has points everywhere locally, but has no $\Q$-rational divisors of odd degree.
		 		\end{Theorem}
		 		
		 		\begin{Remark}
		 			There are examples of smooth plane quartics having points everywhere locally, but no rational divisors of odd degree given in \cite{Bremner}. These examples exploit the fact that the plane quartic in question admits a finite morphism to a genus $1$ curve. As the Jacobian of the curve in Theorem~\ref{thm:example} is absolutely simple, such techniques do not apply.
		 		\end{Remark}
		 	
		 		\begin{proof}[Proof of Theorem~\ref{thm:example}]
		 			$C$ has real points and the polynomial defining $C$ has good reduction at all primes other than $q = 760567$. The point $(0:1948:1) \in C(\F_q)$ is smooth, and for all other primes $p$, $C(\F_p) \ne \emptyset$ (for $p > 37$ this follows from the Weil bounds). So by Hensel's lemma $C$ and, hence, $J^1$ have points everywhere locally. This implies that $\Pic^d(C) = J^d(\Q)$ and $\Pic^d(C_{\Q_p}) = J^d(\Q_p)$ for all primes $p$ and $d \ge 0$.
		 			
		 			Using Magma we compute that $|J(\F_2)| = 25$ and $|J(\F_3)| = 57$. Since these orders are relatively prime, we have $J(\Q)_\textup{tors} = 0$. A search for points of small height on $C$ over $\Q(\sqrt{2})$ yields
		 			\[
		 				D_1 = (\sqrt{2}-2:-\sqrt{2}+1:1) \quad \text{and} \quad D_2 = (-1:\sqrt{2}/2:1)
		 			\]
		 			Then $D = \Tr_{\Q(\sqrt{2})/\Q}(D_1 - D_2)$ is a $\Q$-rational divisor of degree $0$ on $C$ representing a point $P \in J(\Q)$. The image of $P$ under the reduction map $J(\Q) \to J(\F_7)$ is nontrivial. So $P$ has infinite order and, hence, $J(\Q)$ has rank at least $1$.
		 			
		 			To proceed further we compute $\Sel^{f_\frak{m}}_\textup{fake}(J^0)$ and $\Sel^{f_\frak{m}}_\textup{fake}(J^1)$ for the descent setup and modulus setup as in \ref{ex:canonical2d} and \ref{ex:M_canonical}, taking $\Delta$ to be the set of bitangents to $C$ and $\beta$ to be the diagonal embedding of $\Delta$ into $\Div(\Cbar) \times \Delta$. The algebra $L$ has degree $28$ and, moreover, its Galois group is isomorphic to $\operatorname{GSp}_6(\F_2)$, showing that the representation $\Gal_{\Q} \to \operatorname{GSp}(J[2])$ is surjective. A Magma computation (assuming GRH) gives that $\mathcal{O}_L$ has trivial class group. The function $f_{\frak{m}}$ can be written as a ratio of linear forms $f_\frak{m} = l/l_0$, with $l \in L[x,y,z]$ and $l_0 \in \Q[x,y,z]$. Since $\mathcal{O}_L$ has trivial class group, we can scale $l$ by an element of $L^\times$ such that the coefficients of $\ell$ are integral and generate the unit ideal in $\mathcal{O}_L$. 
		 			
		 			By \cite[Theorem 10.9]{BPS}, $\Sel_\textup{fake}^{f_\frak{m}}(J)$ is contained in $L(\calS,2)$, the unramified outside $\calS$ subgroup of $L^\times/\Q^\times L^{\times 2}$ for $\calS = \{2, 760567, \infty\}$. Since $L$ has class number $1$, we can determine representatives in $L^\times$ for $L(\calS,2)$ from the $\calS$-unit group of $L$ (cf. \cite[Proposition 7.3]{BPS}). The order of $J(\Q_p)/2J(\Q_p)$ can be computed from the splitting of $p$ in $L$. This gives an upper bound for the size of the image of $J(\Q_p)$ under $f_\frak{m}$.  For both nonarchimedean primes $p \in \calS$, the differences of images of points in $C(\Q_p)$ already generate a subgroup whose order meets the upper bound, hence must be the image of $J(\Q_p)$. The subgroup of $L(S,2)$ mapping into the images of $J(\Q_p)$ for $p \in \calS$ has $\F_2$-dimension $3$ and contains $\Sel^{f_\frak{m}}_\textup{fake}(J)$. Since the representation $\Gal_{\Q} \to \operatorname{GSp}(J[2])$ is surjective, \cite[Theorem 10.14]{BPS} gives the inequality $\dim_{\F_2}\Sel^2(J) \le \dim_{\F_2}\Sel_\textup{fake}(J) \le 3$.
		 					 			
					The local image $f_\frak{m}(C(\Q_p))$ is unramified for $p$ outside $\calS$ by \cite[Lemma 12.13]{BPS}. Since $f_\frak{m}$ is a homomorphism, the local image $f_\frak{m}(J^1(\Q_p))$ is the coset of $f_\frak{m}(J(\Q_p))$ containing $f_\frak{m}(C(\Q_p))$. It follows that $\Sel_\textup{fake}^{f_\frak{m}}(J^1) \subset L(\calS,2)$. Moreover, $f_\frak{m}(J^1(\Q_p))$ for $p \in \calS$ are easily obtained by translating the $f_\frak{m}(J(\Q_p))$ already computed. It turns out that the image of $L(\calS,2)$ in $(L\otimes \Q_2)^\times/\Q_2^\times (L\otimes \Q_2)^{\times 2}$ does not intersect $f_\frak{m}(J^1(\Q_2))$. Hence, $\Sel^{f_\frak{m}}_\textup{fake}(J^1) = \emptyset$. By Theorem~\ref{thm:selfakeJ1} we have $\Sel^2(J^1) = \emptyset$. In particular, the computation shows that there are no $2$-coverings of $J^1$ with $\Q_2$-points and $\Q_p$-points for all $p$ outside $\calS$.

		 			Since $C$ has points everywhere locally, the  Cassels-Tate pairing on $\Sha(J)$ is alternating by \cite[Corollary 12]{PoonenStoll}. It induces a nondegenerate alternating pairing on the finite group $\frac{\Sha(J)[2]}{2\Sha(J)[4]}$, which consequently has square order (see, for example, \cite[Corollary 4.6]{CreutzANTSX}). The $\Q(\sqrt{2})$-points on $C$ above show that $J^2(\Q) \ne \emptyset$. Since $2[J^1] = [J^2]$ in $\Sha(J)$, we conclude that $[J^1] \in \Sha(J)[2]$. The fact that $\Sel^2(J^1) = \emptyset$ implies, moreover, that $[J^1]$ gives a nontrivial element of $\frac{\Sha(J)[2]}{2\Sha(J)[4]}$ (cf. Theorem~\ref{thm:solublecoverings}\eqref{p7}). We conclude that $\Sha(J)[2^\infty]$ admits a direct summand isomorphic to $\Z/2\Z\times \Z/2\Z$. From the exact sequence
		 				\[
		 					0 \to J(\Q)/2J(\Q) \to \Sel^2(J) \to \Sha(J)[2] \to 0
		 				\]
		 				we therefore obtain that $J(\Q) \simeq \Z$ and $\Sha(J)[2^\infty] \simeq \Z/2\Z\times \Z/2\Z$.
		 		\end{proof}

\subsection{Descent on $J$}

	The following theorem gives an example where we compute $\Sel^2(J)$, despite the fact that \cite[Hypothesis 10.1]{BPS} does not hold.
		
		\begin{Theorem}\label{thm:descentJexample}
			The genus $3$ curve $C \subset \PP^2_\Q$ defined by the vanishing of
			\[
			x^4 + 2x^3y + 2x^3z + 4x^2y^2 + 2x^2yz + 4x^2z^2 + 3xy^3 + 
    2xy^2z + 4xyz^2 + 3xz^3 + 2y^4 + 5y^2z^2 + yz^3 + 2z^4\,
			\]			
			has the following properties.
			\begin{enumerate}
				\item $C(\Q_p) \ne \emptyset$ for all $p \ne 3,\infty$.
				\item $\Pic^{\textup{odd}}(C_{\Q_3}) = \emptyset$. 
				\item $\Pic^{\textup{odd}}(C_\R) = \emptyset$. 
				\item $J = \Jac(C)$ has $\textup{rank}(J(\Q)) = 1$.
			\end{enumerate}
		\end{Theorem}

		\begin{proof}
			The verification of (1) is straightforward. To show $\Pic^{\textup{odd}}(C_{\Q_p}) = \emptyset$ for $p = 3,\infty$ it suffices to check that $C$ has no points over $\Q_p$ or any extension of $\Q_p$ of degree $3$. For this we simply list the finitely many extensions and check locally solubility over each.
			We use the descent setup and modulus setup as in \ref{ex:canonical2d} and \ref{ex:M_canonical}, taking $\Delta$ to be the set of bitangents to $C$ and $\beta$ to be the diagonal embedding of $\Delta$ into $\Div(\Cbar) \times \Delta$. The algebra $L$ has degree $28$ and splits as a product of $2$ quadratic fields (both isomorphic to $\Q(\sqrt{-15})$) and $3$ octic fields. Let $R^\vee = \coker(J[2] \to \Res_\Delta\mu_2/\mu_2)$. As described in \cite[12.6.6]{BPS} we compute the Galois action on the bitangents, from which we find that $\dim_{\F_2}J[2](\Q)  = \dim_{\F_2} R^\vee(\Q) = 2$ and $\dim_{\F_2}(\Res_\Delta\mu_2/\mu_2)(\Q)= 4$. From the exact sequence
			\[
				0 \to J[2](\Q)
 \to (\Res_\Delta\mu_2/\mu_2)(\Q) \to R^\vee(\Q) \to \HH^1(\Q,J[2]) \stackrel{\alpha}\to \HH^1(\Q,\Res_\Delta\mu_2/\mu_2)
			\]
			we conclude that $\alpha$ is injective over $\Q$.
			
			There are points $P_1 = (\eta:1:0),\, P_2 = (\eta:0:1) \in C(\Q(\eta))$, where $\eta$ is a primitive cube root of unity. The torsion subgroup of $J(\Q)$ is $2$-primary, as can be seen by computing $\#J(\F_p)$ for small primes of good reduction. The divisor $D := \Tr_{\Q(\eta)/\Q}(P_1  - P_2)$ represents a point $[D] \in J(\Q)$ which maps to an element of order $18$ in $J(\F_7)$, showing that $[D]$ has infinite order. Thus we have a lower bound $3 \le \dim_{\F_2}\Sel^2(J)$. Moreover, the points $P_i$ show that $\Div^2(C) \ne \emptyset$ and so the conditions of Theorem~\ref{thm:descentJ} are satisfied.

			The curve $C$ has good reduction outside $\calS_1 := \{ 3,5,1613\}$ so by \cite[Theorem 10.9]{BPS}, $\Sel_\textup{fake}^{f_\frak{m}}(J)$ is contained in $L(\calS,2)$, the unramified outside $\calS$ subgroup of $L^\times/k^\times L^{\times 2}$ for $\calS = \{2,3,5,1613, \infty\}$. We compute $L(\calS,2)$ as described in \cite[Proposition 7.3]{BPS}. Since the largest discriminant of a factor of $L$ is of order $10^{28}$, this can be done without assuming GRH. For $p \in \calT = \{ 2,5,1613\}$ we compute the local images $f_\frak{m}(\Pic^0(C_{\Q_p})) = f_\frak{m}(J(\Q_p))$ following the strategy of \cite[Remark 11.6]{BPS} (i.e., compute the images of random points until the dimension of the subgroup they generate meets an upper bound determined in advance from the action of the decomposition group on the bitangents). The subgroup $S_\calT \subset L(\calS,2)$ satisfying these local conditions at all primes in $\calT$ has dimension $5$.
			
			From the action of the decomposition group at $p=3$ on the bitangents we determine that $J(\Q_3)/2J(\Q_3)$ and its image under $\alpha \circ d$ have dimension $3$. However, computing the images of differences of random elements of $\Pic^2(C_{\Q_3})$ under $f_\frak{m}$ we are only able to generate a subgroup $H_3$ of dimension $2$. The subgroup $S_{\calT,H_3}$ of $S_\calT$ restricting to $H_3$ has dimension $2$.
			
			We now consider two cases. If the map $\Pic^0(C_{\Q_3})\to J(\Q_3)/2J(\Q_3)$ is surjective, then $H_3$ has codimension $1$ in $f_\frak{m}(\Pic^0(C_{\Q_3}))$, so $\dim_{\F_2} \Sel_\textup{fake}^{f_\frak{m}}(J) \le \dim_{\F_2} S_{\calT,H_3} + 1 = 3$ and Theorem~\ref{thm:descentJ} applies with $N \le 1$ to give $\dim_{\F_2} \Sel^2(J) \le \dim_{\F_2} \Sel^{f_\frak{m}}_\textup{fake}(J) \le 3$. If the map $\Pic^0(C_{\Q_3})\to J(\Q_3)/2J(\Q_3)$ is not surjective, then $H_3 = f_\frak{m}(\Pic^0(C_{\Q_3}))$, so $\dim_{\F_2} \Sel_\textup{fake}^{f_\frak{m}}(J) \le \dim_{\F_2} S_{\calT,H_3} = 2$ and Theorem~\ref{thm:descentJ} applies with $N \le 2$ to give the upper bound $\dim_{\F_2} \Sel^2(J) \le \dim_{\F_2} \Sel^{f_\frak{m}}_\textup{fake}(J) + 1 \le 3$.
			
			In either case we have the upper bound $\dim_{\F_2} \Sel^2(J) \le 3$ which coincides with the lower bound obtained from the point search. Thus $\textup{rank}(J(\Q)) = 1$.
		\end{proof}

\begin{Remark}
	The computation outlined in the proof above shows that the maps $\Pic^0(C_{\Q_p})\to J(\Q_p)/2J(\Q_p)$ are either surjective for all $p$, or fail to be surjective for both $p = 3$ and $p = \infty$. In fact the latter is the case. To prove this one one can compute $f_\frak{m}(\Pic^0(C_{\Q_p}))$ algorithmically as described in \cite[Section 11.1]{BPS} for either $p = 3$ or $p = \infty$. This shows that the set $S_{\calT,H_3} \simeq \Z/2\Z\times \Z/2\Z$ computed is equal to $\Sel^{f_\frak{m}}_\textup{fake}(J)$. Since $J(\Q)/2J(\Q)$ has dimension $3$ and injects into $\Sel^2(J)$ we conclude that $J(\Q) \ne \Pic^0(C)$.
\end{Remark}

	\section{The set $\Cov_{\frak{m}}^n(C)$ for genus $1$ and hyperelliptic curves}
		Suppose $C$ is a nice curve over $k$ with a modulus setup $(n,\frak{m})$ associated to an $n$-descent setup. In this section we show how the sets $\Cov_{\frak{m}}^n(C)$ and $\Cov_{\frak{m}}^n(J^1)$ generalize known constructions in the situations of \ref{ex:hyp2d} and \ref{ex:g1nd}. For genus $1$ curves this allows us to relate the existence of $\varphi$-coverings to the period-index problem.

	\subsection{Existence of $\varphi$-coverings}
		The following theorem gives, for a modulus setup associated to an $n$-descent setup, several conditions that are equivalent to the existence of an element in $\Cov_{\frak{m}}^n(C)$. 

		\begin{Theorem}
			\label{thm:phicov}
			Suppose $(n,\frak{m})$ is a modulus setup for $C$ associated to an $n$-descent setup $(n,\Delta,\beta)$ and let $\varphi: A_\frak{m} \to J_\frak{m}$ be the isogeny in~\eqref{eq:definephi}. The following are equivalent.
			\begin{enumerate}
				\item\label{p:1} The class of $J^1_\fm$ in $\HH^1(J_\frak{m})$ is divisible by $\varphi$.
				\item\label{p:2} There exists a $\varphi$-covering of $J^1_\frak{m}$.
				\item\label{p:3} There exists a $\varphi$-covering of $C$.
				\item\label{p:4'} $\Cov^\varphi(C) \ne \emptyset$.
				\item\label{p:3'} $\Cov_{\frak{m}}^n(C) \ne \emptyset$.
				\item\label{p:3''} $\Cov_{\frak{m}}^n(J^1) \ne \emptyset$.
				\item\label{p:4} There exists an $n$-covering $\pi:X \to C$ with the property that $\pi^*\beta_\delta$ is linearly equivalent to a $k$-rational divisor, for some $\delta \in \Delta(\kbar)$.
				\item\label{p:5} The maximal unramified abelian covering of $C_\kbar$ of exponent $n$ descends to $k$ and the image of the $k$-rational divisor class $\pi^*\beta_\delta$ in $\Br(k)$ under the map $\Theta_{X}$ of~\eqref{eq:defThetaC} lies in the image of the map $\Upsilon$ of~\eqref{eq:maindiagram}, for every maximal unramified abelian covering $\pi:X\to C$ of exponent $n$ and every $\delta \in \Delta(\kbar)$.
			\end{enumerate}
		\end{Theorem}
		
		Before giving the proof we state and prove two lemmas.

		\begin{Lemma}\label{lem:phicov}
			Suppose $(n,\frak{m})$ is a modulus setup associated to an $n$-descent setup $(n,\Delta,\beta)$ and that $\pi:X \to C$ is an $n$-covering. The class of $(X,\pi)$ in $\Cov^n(C)$ lies in $\Cov_{\frak{m}}^n(C)$ if and only if $\pi^*\beta_\delta$ is linearly equivalent to a $k$-rational divisor, for some $\delta \in \Delta(\kbar)$.
		\end{Lemma}
			
		\begin{proof}
			Suppose $\pi:X \to C$ lifts to a $\varphi$-covering $Y \to C$. The subfield $k(X) \subset k(Y)$ corresponds to the subgroup $\mu_n = T'[\varphi]\subset A_\frak{m}[\varphi]$. The extension $k(X) \subset k(Y)$ is therefore obtained by adjoining to $k(X)$ an $n$-th root of a function $f$ such that $\divv(f) = nD - \pi^*d\frak{m}$, for some $d \in \Z$ and $f \in k(X)^\times$. Furthermore, we can arrange that $d = 1$. Indeed, we must have $\gcd(n,d) = 1$, otherwise there would be a proper unramified intermediate extension of $k(X) \subset k(Y)$. Hence $\pi^*\frak{m} = nD + \divv(f)$ for some $D \in \Div(X)$ and $f \in k(X)^\times$. Recall that $n\beta - \frak{m}\times \Delta = \divv(f_\frak{m})$. So, for any $\delta \in \Delta(\kbar)$, the function $h := f/\pi^*(f_{\frak{m},\delta}) \in \kbar(X_\kbar)^\times$ has divisor $n(D - \pi^*\beta_\delta)$. Since adjoining an $n$th root of $h$ to $\kbar(X_\kbar)$ gives an unramified intermediate field of $\kbar(X_\kbar) \subset \kbar(Y_\kbar)$, we must have $h \in \kbar(X_\kbar)^{\times n}$. This shows that $D - \pi^*\beta_\delta$ is principal.
					
		For the other direction, suppose $D \in \Div(X)$ is a $k$-rational divisor linearly equivalent to $\pi^*\beta_\delta$. Then $\divv(\pi^*f_{\frak{m},\delta}) = n\pi^*\beta_\delta-\pi^*\frak{m} = nD - \pi^*\frak{m} + \divv(f)$, for some $f \in \kbar(X_\kbar)^\times$. Thus, the divisor $nD - \pi^*\frak{m} \in \Div(X)$ is principal and $k$-rational. By Hilbert's Theorem 90 it is the divisor of some $k$-rational function $g \in k(X)^\times$. Let $Y \to X$ be the covering obtained by adjoining an $n$-th root of $g$ to $k(X)$. Over $\kbar$ we see that $\kbar(Y)$ is the compositum of $\kbar(X_\kbar)$ and $\kbar(C_\kbar)(\sqrt[n]{f_{\frak{m},\delta}})$, so $Y \to C$ is a $\varphi$-covering of $C$.
		\end{proof}

	\begin{Lemma}\label{lem:twistcomputation}
				Suppose $\pi:X\to C$ is an $n$-covering and $\pi_z:X_z \to C$ is the twist by the cocycle $z \in Z^1(J[n])$. Let $\Theta_X$ and $\Theta_{X_z}$ denote the maps from~\eqref{eq:defThetaC} and let $\Upsilon$ denote the map in~\eqref{eq:maindiagram}. For any $\delta \in \Delta(\kbar)$, 
				\[
					\Upsilon([z]) = \Theta_{X_z}(\pi_z^*\beta_\delta) - \Theta_X(\pi^*\beta_\delta).
				\]
			\end{Lemma}
		
			\begin{proof}
				There is an isomorphism of coverings $\rho : \Xbar_z \to \Xbar$ with the property that ${ }^{\sigma}\rho\circ\rho^{-1} = T_{z_\sigma} \in \Aut(\Xbar/C_\kbar)$ is translation by $z_\sigma \in J[n]$, for every $\sigma \in \Gal_k$. Let $W = \pi_z^*\beta_\delta$ and $W' := \rho^*(W) = \pi^*\beta_\delta$. These represent Galois invariant divisor classes, hence, for any $\sigma \in \Gal_k$ there are functions $f_\sigma \in \kbar(X_z)^\times$ and $g_\sigma \in \kbar(X)^\times$ with $\divv(f_\sigma) = {}^\sigma W - W$ and $\divv(g_\sigma) = {}^\sigma W' - W'$. The classes in $\Br(k)$ of $W$ and $W'$ are given by the $2$-cocycles
			\[
				a_{(\sigma,\tau)}=\frac{{}^\sigma f_\tau \cdot f_\sigma}{f_{\sigma\tau}}
				\quad\text{and}\quad
				a'_{(\sigma,\tau)}=\frac{{}^\sigma g_\tau \cdot g_\sigma}{g_{\sigma\tau}}\,,
			\]
			both of which take values in $\kbar^\times$. Since $f_\sigma/\rho^*g_\sigma \in \kbar^\times$, the computation
			\[
				\frac{a_{(\sigma,\tau)}}{a'_{(\sigma,\tau)}} 
				= \frac{a_{(\sigma,\tau)}}{\rho^*(a'_{(\sigma,\tau)})}
				= \underbrace{{}^\sigma\left(\frac{f_\tau}{\rho^* g_\tau}\right)
					\cdot \frac{f_\sigma}{\rho^*g_\sigma}
					\cdot \frac{\rho^*g_{\sigma\tau}}{f_{\sigma\tau}} }_{\text{coboundary}}
					\cdot  \frac{{}^\sigma(\rho^*g_\tau)}{\rho^*({}^\sigma g_\tau)}
			\]
			shows that $\Theta_{X_z}(W)-\Theta_{X}(W')$ is represented by the $2$-cocycle $\eta \in Z^2(\Gal_k,\kbar^\times)$ defined by
			\[
				\eta_{(\sigma,\tau)} = \frac{{}^\sigma(\rho^*g_\tau)}{\rho^*({}^\sigma g_\tau)} =  \frac{{}^\sigma g_\tau\circ{}^\sigma\rho}{{}^\sigma g_\tau \circ \rho}\,.
			\]
			Using that $(\rho^{-1})^*$ is the identity on $\kbar \subset \kbar(Y)$ and that ${}^\sigma\rho\circ\rho^{-1} = T_{z_\sigma}$ we have $\eta_{(\sigma,\tau)}= \frac{{}^\sigma g_\tau \circ T_{z_\sigma}}{{}^\sigma g_\tau}\,.$ We recognize this as the Weil pairing $\eta_{(\sigma,\tau)} = e_n({}^\sigma P_\tau,z_\sigma)$, where $P_\tau \in J[n]$ is the class represented by the divisor ${}^\tau\beta_\delta-\beta_\delta$ (see Lemma~\ref{lem:WP}). The cocycle $P_\tau \in Z^1(\Gal_k,J[n])$ represents $\partial(1)$ where $\partial$ is the coboundary map in~\eqref{eq:maindiagram}. So $\eta_{(\sigma,\tau)}$ represents the $e$-pairing cup product $\partial(1) \cup_e [z] = [z] \cup_e \partial(1) = \Upsilon([z])$ by Lemma~\ref{lem:WPcupprod}.
			\end{proof}

			\begin{proof}[Proof of Theorem~\ref{thm:phicov}]
				There exists a $\varphi$-covering of $(J^1_\frak{m})_\kbar$. The Galois descent obstruction to defining this over $k$ is the image in $\HH^2(k,A_\frak{m}[\varphi])$ of the class of this covering under the map
				\[
					\HH^0\left(\Gal_k,\HH^1\left((J^1_\frak{m})_\kbar,A_\frak{m}[\varphi]\right)\right) \to \HH^2(\Gal_k,A_\frak{m}[\varphi])
				\]
				from the Hochschild-Serre spectral sequence (cf. \cite[Section 2.2]{Skorobogatov}). This class coincides with the image of $[J^1_\frak{m}]$ under the coboundary map arising from the exact sequence
				\[
					0 \to A_\frak{m}[\varphi] \to A_\frak{m} \to J_\frak{m} \to 0
				\]
				(see \cite[Lemma 2.4.5]{Skorobogatov}). This proves the equivalence of \reff{p:1} and \reff{p:2}, while the equivalence of \reff{p:2} and \reff{p:3} follows from geometric class field theory. The equivalences \reff{p:3} $\Leftrightarrow$ \reff{p:4'} $\Leftrightarrow$ \reff{p:3'} $\Leftrightarrow$ \reff{p:3''} follow immediately from the definitions, and \reff{p:3} $\Leftrightarrow$ \reff{p:4} is given by Lemma~\ref{lem:phicov}.

				It remains to prove \reff{p:4} $\Leftrightarrow$ \reff{p:5}. An $n$-covering $\pi:X \to C$ is a $k$-form of the maximal unramified abelian covering of exponent $n$, which we may assume exists. Then, for any $\delta,\delta' \in \Delta(\kbar)$ the divisors $\pi^*\beta_\delta$ and $\pi^*\beta_{\delta'}$ are linearly equivalent. Indeed $\beta_\delta-\beta_{\delta'}$ represents a class in $J[n]$. It follows that the class of $\pi^*\beta_\delta$ in $\Pic(X_\kbar)$ is fixed by $\Gal_k$. The image of this class in $\Br(k)$ is trivial if and only if the class can be represented by a $k$-rational divisor. Since the set of all isomorphism classes of $n$-coverings of $C$ is a principal homogeneous space for $\HH^1(J[n])$ under the action of twisting, the equivalence of \reff{p:4} and  \reff{p:5} follows from Lemma~\ref{lem:twistcomputation}.
			\end{proof}

		\subsection{Hyperelliptic curves}

			Suppose $(2,\frak{m})$ is a modulus setup for $C : z^2 = f(x,y)$, a double cover of $\PP^1$ as in \ref{ex:hyp2d}. 
			\begin{enumerate}
				\item Given a pair of symmetric bilinear forms $(A,B)$ such that $\disc(Ax-By) = f(x,y)$ the Fano variety of maximal linear subspaces contained in the base locus of the pencil of quadrics generated by $(A,B)$ may be given the structure of a $2$-covering of $J^1$. Theorem 22 and the discussion of Section 5 in \cite{BGW} shows that the isomorphism classes of $2$-coverings of $J^1$ that arise in this way are precisely those in $\Cov_{\frak{m}}^2(J^1)$. 
				\item Section 3 of \cite{BruinStoll} gives an explicit construction of a collection of $2$-coverings of $C$ from the set $H_k$ (notation as in \cite{BruinStoll}). Comparing Lemma~\ref{lem:phicov} with the proof of \cite[Theorem 3.4]{BruinStoll} shows that the collection of coverings they produce is precisely $\Cov_{\frak{m}}^2(C)$.
				\item In \cite[Section 6]{CreutzANTSX} a set $\Cov_{\textup{good}}(J^1/k)$ is defined; from that definition and point (2) above it follows that this set coincides with $\Cov_{\frak{m}}^2(J^1)$. See also~\cite[Lemma 2.3]{CreutzIJNT} for a direct proof that $\Cov_{\textup{good}}(J^1/k)$ coincides with the set described in (1) above.
			\end{enumerate}

\subsection{Genus $1$ curves}\label{sec:g1}
		 	For a genus $1$ curve $C$ there is a natural identification $C = J^1$, and $C$ can be endowed with the structure of a torsor under its Jacobian $J$. We define the \defi{index} of $C$ to be the least positive degree of a $k$-rational divisor on $C$ and the \defi{period} of $C$ to be the order of the class $[C]$ in $\HH^1(E)$. The index $I$ and period $P$ of $C$ are known to satisfy $P \mid I \mid P^2$, and over number fields all pairs of integers $(P,I)$ satisfying these relations are known to occur \cite{CSpotentialSha}. The following result gives an interpretation of the equivalent conditions of Theorem~\ref{thm:phicov} in terms of period and index of the $n$-coverings of $C$.
		 	
		 	The proof of the following theorem is given at the end of this section.
		 		
		 		\begin{Theorem}\label{thm:PIg1}
		 			Let $[C]$ be a torsor under an elliptic curve $E$ with underlying curve $C$. The following are equivalent.
		 			\begin{enumerate}
		 				\item\label{q1} There exists a torsor $[C'] \in \HH^1(E)$ of index dividing $n^2$ such that $n[C']=[C]$.
		 				\item\label{q2} The curve $C$ admits a modulus setup $(n,\frak{m})$ with $n = \deg(\frak{m})$ such that $[J^1_\frak{m}]$ is divisible by $\varphi$ in $\HH^1(J_\frak{m})$.
		 			\end{enumerate}
		 		\end{Theorem}
		 		 
				\begin{Remark}
		 			In \cite{CreutzPAMS} it is shown that condition \reff{q2} is satisfied when $C$ is a locally soluble curve over a global field $k$ and the action of $\Gal_k$ on $J[n]$ is sufficiently generic. In particular, when $k = \Q$, it holds when $n = p^r$ is any prime power with $p > 7$.
		 		\end{Remark}

				From the proof one extracts the following, which shows that the set $\Cov_{\frak{m}}^n(C)$ of this paper coincides with the set $\Cov_0^n(C)$ defined in~\cite[Definition 3.3]{CreutzMathComp}.
	
				\begin{Corollary}
					Let $C$ be a genus $1$ curve with a modulus setup $(n,\frak{m})$ with $n = \deg(\frak{m})$. The set $\Cov_{\frak{m}}^n(C)$ consists of those $n$-coverings $D \to C$ such that the index of $D$ divides $n^2$.
				\end{Corollary}	
		 		
		 		Our proof of Theorem~\ref{thm:PIg1} will make use of the following interpretation of the elements of $\HH^1(E[n])$ taken from \cite{CFOSS}.
		 		
		 		\begin{Definition}
		 		  A \defi{torsor divisor class pair} $(T,Z)$ consists of a $E$-torsor $T$ and a $k$-rational divisor class $Z\in \Pic_T(k)$. Two torsor divisor class pairs $(T,Z)$ and $(T',Z')$ are isomorphic if there is an isomorphism of torsors $s:T \to T'$ such that $s^*Z' = Z$. 
				\end{Definition}
				
				The automorphism group of the pair $(E,n.0_E)$ can be identified with $E[n]$, and every pair $(T,Z)$ with $\deg(Z) = n$ can be viewed as a twist of $(E,n.0_E)$ (\cite[Lemmas 1.7 and 1.8]{CFOSS}). It follows that the torsor divisor class pairs of degree $n$, viewed as twists of $(E,n.0_E)$, are parameterized by the group $\HH^1(E[n])$.
				
				\begin{Lemma}\label{lem:oblemma}
					Suppose $(T',Z')$ is a torsor divisor class pair representing a lift of the class of $(T,Z)$ under the map $n_*:\HH^1(E[n^2]) \to \HH^1(E[n])$. The Brauer classes associated to the $k$-rational divisor classes $Z'$ and $Z$ satisfy $n[Z'] = [Z]$ in $\Br(k)$. In particular, $Z$ is represented by a $k$-rational divisor if $Z'$ is.
				\end{Lemma}
				
				\begin{proof}
					Suppose the class of $(T',Z')$ is represented by a $1$-cocycle $\xi_\sigma \in Z^1(E[n^2])$.  Let $f_\sigma, g_\sigma \in \kbar(E)^\times$ be functions such that $\divv(f_\sigma) = \tau_{\xi_\sigma}^*[n]^*0_E - [n]^*0_E$ and $\divv(g_\sigma) = \tau_{n\xi_\sigma}^*n.0_E - n.0_E$. Comparing divisors we see that we may scale by a constant to arrange that $f_\sigma^n = g_\sigma \circ [n]$. Moreover, using that $\xi_\sigma$ is a cocycle, we see that the coboundaries of the $1$-cochains $(\sigma \mapsto f_\sigma)$ and $(\sigma \mapsto g_\sigma)$ give $2$-cocycles $F,G \in Z^2(\kbar^\times)$ satisfying $F^n = G$.
					
					To prove the lemma one shows that $F$ and $G$ represent the Brauer classes corresponding to $Z'$ and $Z$, respectively. By \cite[Prop. 1.32]{CFOSS}, the pair $(g_\sigma,n\xi_\sigma)$ denotes a lift of $n\xi_\sigma$ to the theta group corresponding to the torsor divisor class pair $(E,n.0_E)$. Then \cite[Prop. 2.2]{CFOSS} shows that $[G] = [Z]$. In the same way we see that $(f_\sigma,\xi_\sigma)$ gives a lift of $\xi_\sigma$ to the theta group corresponding to $(E,[n]^*0_E) \simeq (E,n^2.0_E)$ and so $[F] = [Z']$.
				\end{proof}

		 		\begin{proof}[Proof of Theorem~\ref{thm:PIg1}]
		 		
		 			We may assume $n > 1$. 
		 			
		 			\reff{q1} $\Rightarrow$ \reff{q2}. Suppose \reff{q1} holds and let $Z' \in \Pic^{n^2}(C')$. Consider the torsor divisor class pair $([C'],Z')$. The image of this class under $n_*:\HH^1(E[n^2]) \to \HH^1(E[n])$ is represented by a pair $([C],Z)$. By Lemma~\ref{lem:oblemma}, $Z \in \Pic^n(C)$. By Riemann-Roch $Z$ determines a map $C \to \PP^{n-1}$ (which is an embedding for $n > 2$ and a double cover for $n = 2$). By Bertini the divisor class $Z$ contains a reduced and effective and base point free divisor $\frak{m}$ of degree $n$. Then $(n,\frak{m})$ is a modulus setup for $C$ with $n = \deg(\frak{m})$. Let $\Delta := \{ x \in C(\kbar) \;:\; n.x \sim \frak{m} \}$ and take $\beta$ to be the diagonal embedding of $\Delta$ in $C \times \Delta$. Then $(n,\frak{m})$ is associated to the $n$-descent setup $(n,\Delta,\beta)$, which agrees with that described in \ref{ex:g1nd}.

		 			The pair $(C',Z')$ corresponds to an $n^2$-covering of $E$, which we may assume factors through the $n$-covering of $E$ determined by $(C,Z)$. In particular, there is a commutative diagram
		 			\[
		 				\xymatrix{
		 					C' \ar[r]^{\pi'}\ar@{.>}[d]^{s'} & C \ar[r]^{\pi} \ar@{.>}[d]^s & E \ar@{=}[d]\\
		 					E \ar[r]^n & E\ar[r]^n&E
		 				}
		 			\]
		 			where $s$ and $s'$ are isomorphisms defined over $\kbar$ which determine the $E$-torsor structures on $C$ and $C'$. Now $[\frak{m}] = Z = [s^*n.0_E]$, so we must have $s^*0_E = \beta_\delta$ for some $\delta \in \Delta(\kbar)$. On the other hand, $Z'$ is the class of $s'^*n^2.0_E = s'^*[n]^*0_E = \pi'^*s^*0_E = \pi'^*\beta_\delta$. As this class is represented by a $k$-rational divisor, Theorem~\ref{thm:phicov} shows that $[J^1_\frak{m}]$ is divisible by $\varphi$.

		 			\reff{q2} $\Rightarrow$ \reff{q1}. Then $\frak{m}$ is ample and base point free and, hence, determines a model of $C$ as a degree $n$ curve in $\PP^{n-1}$. Let $(n,\Delta,\beta)$ be the $n$-descent setup as in \ref{ex:g1nd}. By Theorem~\ref{thm:phicov} there is an $n$-covering $\pi:C' \to C$ such that $\pi^*\beta_\delta$ is linearly equivalent to a $k$-rational divisor for some $\delta \in \Delta(\kbar)$. The genus $1$ curve $C'$ is  endowed with a torsor structure so that $n[C'] = [C]$ in $\HH^1(E)$. Moreover, the index of $[C']$ divides $\deg(\pi^*\beta_\delta) = n^2$. 			
		 		\end{proof}



\begin{bibdiv}
		\begin{biblist}

\bib{Atiyah}{article}{
   author={Atiyah, Michael F.},
   title={Riemann surfaces and spin structures},
   journal={Ann. Sci. \'Ecole Norm. Sup. (4)},
   volume={4},
   date={1971},
   pages={47--62},
   issn={0012-9593},
}

\bib{Bhargava}{article}
  {
   author={Bhargava, Manjul},
   title={Most hyperelliptic curves over $\Bbb Q$ have no rational points},
   eprint={arXiv:1308.0395}
   }
   

\bib{BG}{article}{
   author={Bhargava, Manjul},
   author={Gross, Benedict H.},
   title={The average size of the 2-Selmer group of Jacobians of
   hyperelliptic curves having a rational Weierstrass point},
   conference={
      title={Automorphic representations and $L$-functions},
   },
   book={
      series={Tata Inst. Fundam. Res. Stud. Math.},
      volume={22},
      publisher={Tata Inst. Fund. Res., Mumbai},
   },
   date={2013},
   pages={23--91},
}

\bib{BGW}{article}{
   author={Bhargava, Manjul},
   author={Gross, Benedict H.},
   author={Wang, Xiaoheng},
   title={A positive proportion of locally soluble hyperelliptic curves over
   $\Bbb Q$ have no point over any odd degree extension},
   note={With an appendix by Tim Dokchitser and Vladimir Dokchitser},
   journal={J. Amer. Math. Soc.},
   volume={30},
   date={2017},
   number={2},
   pages={451--493},
   issn={0894-0347},
}
		

\bib{BGW_AIT2}{article}{
   author={Bhargava, Manjul},
   author={Gross, Benedict H.},
   author={Wang, Xiaoheng},
   title={Arithmetic invariant theory II: Pure inner forms and obstructions
   to the existence of orbits},
   conference={
      title={Representations of reductive groups},
   },
   book={
      series={Progr. Math.},
      volume={312},
      publisher={Birkh\"auser/Springer, Cham},
   },
   date={2015},
   pages={139--171},
}
		 
%
%

\bib{BSD}{article}{
   author={Birch, B. J.},
   author={Swinnerton-Dyer, H. P. F.},
   title={Notes on elliptic curves. I},
   journal={J. Reine Angew. Math.},
   volume={212},
   date={1963},
   pages={7--25},
   issn={0075-4102},
}



\bib{Borovoi}{article}{
   author={Borovoi, Mikhail},
   title={The Brauer-Manin obstructions for homogeneous spaces with
   connected or abelian stabilizer},
   journal={J. Reine Angew. Math.},
   volume={473},
   date={1996},
   pages={181--194},
   issn={0075-4102},
}


\bib{magma}{article}{
  author={Bosma, W.},
  author={Cannon, J.},
  author={Playoust, C.},
  title={The Magma algebra system. I. The user language}, 
  journal={J. Symbolic Comput.},
  volume={24},
  date={1997},
  pages={235--265}
  }

\bib{Bremner}{article}{
   author={Bremner, Andrew},
   title={Some quartic curves with no points in any cubic field},
   journal={Proc. London Math. Soc. (3)},
   volume={52},
   date={1986},
   number={2},
   pages={193--214},
   issn={0024-6115},
}

\bib{BPS}{article}{
   author={Bruin, Nils},
   author={Poonen, Bjorn},
   author={Stoll, Michael},
   title={Generalized explicit descent and its application to curves of genus $3$},
   journal={Forum of Mathematics, Sigma},
   date={2016},
   volume={4},
   issn = {2050-5094},
   pages={e6 80 pages},
}



\bib{BruinStoll}{article}{
   author={Bruin, Nils},
   author={Stoll, Michael},
   title={Two-cover descent on hyperelliptic curves},
   journal={Math. Comp.},
   volume={78},
   date={2009},
   number={268},
   pages={2347--2370},
   issn={0025-5718},
}

\bib{Cassels}{article}{
   author={Cassels, J. W. S.},
   title={Arithmetic on curves of genus $1$. IV. Proof of the
   Hauptvermutung},
   journal={J. Reine Angew. Math.},
   volume={211},
   date={1962},
   pages={95--112},
   issn={0075-4102},
}

\bib{CSpotentialSha}{article}{
   author={Clark, Pete L.},
   author={Sharif, Shahed},
   title={Period, index and potential. III},
   journal={Algebra Number Theory},
   volume={4},
   date={2010},
   number={2},
   pages={151--174},
   issn={1937-0652},
}

\bib{Cremona}{article}{
   author={Cremona, J. E.},
   title={Classical invariants and 2-descent on elliptic curves},
   note={Computational algebra and number theory (Milwaukee, WI, 1996)},
   journal={J. Symbolic Comput.},
   volume={31},
   date={2001},
   number={1-2},
   pages={71--87},
   issn={0747-7171},
}

\bib{CFOSS}{article}{
   author={Cremona, J. E.},
   author={Fisher, T. A.},
   author={O'Neil, C.},
   author={Simon, D.},
   author={Stoll, M.},
   title={Explicit $n$-descent on elliptic curves. I. Algebra},
   journal={J. Reine Angew. Math.},
   volume={615},
   date={2008},
   pages={121--155},
   issn={0075-4102},
}

\bib{CreutzPhD}{article}{
   author={Creutz, Brendan},
   title={Explicit second p-descent on elliptic curves},
   note={Ph.D. thesis, Jacobs University},
   date={2010}
}

		
%
		
\bib{CreutzANTSX}{article}{
   author={Creutz, Brendan},
   title={Explicit descent in the Picard group of a cyclic cover of the projective line},
   book={
     title={Algorithmic number theory: Proceedings of the 10th Biennial International Symposium (ANTS-X) held in San Diego, July 9--13, 2012},
     series={Open Book Series},
     volume={1},
     publisher={Mathematical Science Publishers},
     editor={Everett W. Howe},	
     editor={Kiran S. Kedlaya}			    
   },
  date={2013},
  pages={295--315}
}

\bib{CreutzMathComp}{article}{
   author={Creutz, Brendan},
   title={Second $p$-descents on elliptic curves},
   journal={Math. Comp.},
   volume={83},
   date={2014},
   number={285},
   pages={365--409},
   issn={0025-5718},
}

\bib{CreutzPAMS}{article}{
   author={Creutz, Brendan},
   title={Most binary forms come from a pencil of quadrics},
   journal={Proc. Amer. Math. Soc. Ser. B},
   volume={3},
   date={2016},
   pages={18--27},
   issn={2330-1511},
}

\bib{CreutzIJNT}{article}{
	author={Creutz, Brendan},
	title={Improved rank bounds from $2$-descent on hyperelliptic Jacobians},
	journal={Int. J. Number Theory},
	volume={14},
	number={(6)},
	pages={1709--1713},
	date={2018}
	}
	
\bib{CreutzViray}{article}{
   author={Creutz, Brendan},
   author={Viray, Bianca},
   title={Two torsion in the Brauer group of a hyperelliptic curve},
   journal={Manuscripta Math.},
   volume={147},
   date={2015},
   number={1-2},
   pages={139--167},
   issn={0025-2611},
}

\bib{Grothendieck}{article}{
   author={Grothendieck, Alexander},
   title={Techniques de construction et th\'eor\`emes d'existence en g\'eom\'etrie
   alg\'ebrique. III. Pr\'eschemas quotients},
   language={French},
   conference={
      title={S\'eminaire Bourbaki, Vol. 6},
   },
   book={
      publisher={Soc. Math. France, Paris},
   },
   date={1995},
   pages={Exp. No. 212, 99--118},
}

\bib{Howe}{article}{
   author={Howe, Everett W.},
   title={The Weil pairing and the Hilbert symbol},
   journal={Math. Ann.},
   volume={305},
   date={1996},
   number={2},
   pages={387--392},
   issn={0025-5831},
}



\bib{Lichtenbaum}{article}{
   author={Lichtenbaum, Stephen},
   title={Duality theorems for curves over $p$-adic fields},
   journal={Invent. Math.},
   volume={7},
   date={1969},
   pages={120--136},
   issn={0020-9910},
}

\bib{MSS}{article}{
   author={Merriman, J. R.},
   author={Siksek, S.},
   author={Smart, N. P.},
   title={Explicit $4$-descents on an elliptic curve},
   journal={Acta Arith.},
   volume={77},
   date={1996},
   number={4},
   pages={385--404},
   issn={0065-1036},
}


\bib{Milne_jac}{article}{
   author={Milne, J. S.},
   title={Jacobian varieties},
   conference={
      title={Arithmetic geometry},
      address={Storrs, Conn.},
      date={1984},
   },
   book={
      publisher={Springer, New York},
   },
   date={1986},
   pages={167--212},
}
		

\bib{Milne_AV}{article}{
   author={Milne, J. S.},
   title={Abelian varieties},
   conference={
      title={Arithmetic geometry},
      address={Storrs, Conn.},
      date={1984},
   },
   book={
      publisher={Springer, New York},
   },
   date={1986},
   pages={103--150},
}
	
%

\bib{MumfordAV}{book}{
   author={Mumford, David},
   title={Abelian varieties},
   series={Tata Institute of Fundamental Research Studies in Mathematics,
   No. 5 },
   publisher={Published for the Tata Institute of Fundamental Research,
   Bombay; Oxford University Press, London},
   date={1970},
   pages={viii+242},
}


\bib{nsw}{book}{
   author={Neukirch, J{\"u}rgen},
   author={Schmidt, Alexander},
   author={Wingberg, Kay},
   title={Cohomology of number fields},
   series={Grundlehren der Mathematischen Wissenschaften [Fundamental
   Principles of Mathematical Sciences]},
   volume={323},
   edition={2},
   publisher={Springer-Verlag, Berlin},
   date={2008},
   pages={xvi+825},
   isbn={978-3-540-37888-4},
}

\bib{PoonenRains}{article}{
   author={Poonen, Bjorn},
   author={Rains, Eric},
   title={Self cup products and the theta characteristic torsor},
   journal={Math. Res. Lett.},
   volume={18},
   date={2011},
   number={6},
   pages={1305--1318},
   issn={1073-2780},
}


\bib{PoonenSchaefer}{article}{
   author={Poonen, Bjorn},
   author={Schaefer, Edward F.},
   title={Explicit descent for Jacobians of cyclic covers of the projective
   line},
   journal={J. Reine Angew. Math.},
   volume={488},
   date={1997},
   pages={141--188},
   issn={0075-4102},
}



\bib{PoonenStoll}{article}{
   author={Poonen, Bjorn},
   author={Stoll, Michael},
   title={The Cassels-Tate pairing on polarized abelian varieties},
   journal={Ann. of Math. (2)},
   volume={150},
   date={1999},
   number={3},
   pages={1109--1149},
   issn={0003-486X},
}
%
%
		




\bib{SerreAGCF}{book}{
   author={Serre, Jean-Pierre},
   title={Algebraic groups and class fields},
   series={Graduate Texts in Mathematics},
   volume={117},
   note={Translated from the French},
   publisher={Springer-Verlag, New York},
   date={1988},
   pages={x+207},
   isbn={0-387-96648-X},
}


\bib{Skorobogatov}{book}{
   author={Skorobogatov, Alexei},
   title={Torsors and rational points},
   series={Cambridge Tracts in Mathematics},
   volume={144},
   publisher={Cambridge University Press, Cambridge},
   date={2001},
   pages={viii+187},
   isbn={0-521-80237-7},
}

\bib{Sutherland}{article}{
  author={Sutherland, Andrew},
  title={A database of nonhyperelliptic genus 3 curves over $\Q$},
  eprint={arXiv:1806.06289},
  date={2018}
}


\bib{Thorne1}{article}{
   author={Thorne, Jack A.},
   title={$E_6$ and the arithmetic of a family of nonhyperelliptic
   curves of genus 3},
   journal={Forum Math. Pi},
   volume={3},
   date={2015},
   pages={e1, 41},
   issn={2050-5086},
}

\bib{Thorne2}{article}{
   author={Thorne, Jack A.},
   title={On the 2-Selmer groups of plane quartic curves with a marked point},
   note={(preprint)}
}

\bib{Wang}{article}{
   author={Wang, Xiaoheng},
   title={Maximal linear spaces contained in the based loci of pencils of
   quadrics},
   journal={Algebr. Geom.},
   volume={5},
   date={2018},
   number={3},
   pages={359--397},
   issn={2214-2584},
}



	
			\end{biblist}
	\end{bibdiv}
\end{document}